\documentclass[12pt]{article}

\usepackage[hmargin=1.0in,vmargin=1.3in]{geometry}
\usepackage{amsmath,amssymb,amsthm}
\usepackage{amsrefs}
\usepackage{hyperref}

\newcommand{\wkl}{{\sf{WKL}}_0}
\newcommand{\rca}{{\sf{RCA}}_0}
\newcommand{\aca}{{\sf{ACA}}_0}
\newcommand{\seq}[2]{{\langle #1 _ {#2} \rangle_{#2 \in \nat}}}
\newcommand{\nat}{{\mathbb N}}
\newcommand{\rat}{{\mathbb Q}}
\newcommand{\lth}{{\text{lh}}}

\newcommand{\nor}[1]{{\sf { NOR{#1}}}}
\newcommand{\gal}{{\sf {GAL}}}
\newcommand{\gfp}{{\sf{GF}}(p)}
\newcommand{\gfpn}{{\sf{GF}}(p^n)}
\newcommand{\clo}{\overline}

\newcommand{\la}{\langle}
\newcommand{\ra}{\rangle}
\newcommand{\Nb}{\mathbb{N}}
\newcommand{\Qb}{\mathbb{Q}}

\newcommand{\bQb}{\clo{\mathbb{Q}}}
\newcommand{\imp}{\rightarrow}

\newcommand{\biimp}{\leftrightarrow}

\newcommand{\Aut}{\mathop{\mathrm{Aut}}\nolimits}
\newcommand{\Sym}{\mathop{\mathrm{Sym}}\nolimits}

\DeclareMathOperator{\ran}{ran}

\theoremstyle{plain}
\newtheorem{thm}{Theorem}
\newtheorem{cor}[thm]{Corollary}  
\newtheorem{lemma}[thm]{Lemma}

\newtheorem*{lemma*}{Lemma}
\newtheorem*{thm*}{Theorem}

\theoremstyle{definition}
\newtheorem{defn}[thm]{Definition}

\newtheorem{question}[thm]{Question}

\begin{document}             

\title{Reverse Mathematics and Algebraic Field Extensions}
\author{
Fran\c{c}ois G. Dorais
\and 
Jeffry Hirst
\and
Paul Shafer
}
\date{September 3, 2012\\{\small(Revised May 10, 2013)}}
\maketitle

\begin{abstract}
This paper analyzes theorems about algebraic field extensions using the techniques of reverse mathematics.
In section \S\ref{sectionauto}, we show that $\wkl$ is equivalent to the ability to extend $F$-automorphisms of field extensions to
automorphisms of $\clo F$, the algebraic closure of $F$.  Section \S\ref{sectioniso} explores finitary conditions for embeddability.
Normal and Galois extensions are discussed in section \S\ref{sectionnorm}, and the Galois
correspondence theorems for infinite field extensions are treated in section \S\ref{sectiongalcor}.
\end{abstract}

Reverse mathematics is a foundational program in which mathematical theorems are analyzed
using a hierarchy of subsystems of second order arithmetic.  This paper uses three such subsystems.  The
base system $\rca$ includes $\Sigma^0_1$-$\sf{IND}$ (induction for $\Sigma^0_1$ formulas) and set comprehension
for $\Delta^0_1$ definable subsets of $\nat$.  The stronger system $\wkl$ appends K\"o{}nig's theorem restricted
to binary trees (subtrees of $2^{<\Nb}$).  The even stronger system $\aca$ adds comprehension for
arithmetically definable subsets of $\nat$.  For a detailed formulation of these subsystems and related analysis
of many mathematical theorems, see Simpson's book \cite{simpson}. 

Reverse mathematics of countable algebra, including topics from group theory, ring theory, and field theory, can be found in the paper of
Friedman, Simpson, and Smith \cite{fss}.  Further discussion appears throughout Simpson's book \cite{simpson}.  A field
is a set of natural numbers with operations and constants satisfying the field axioms.  Field embeddings and isomorphisms
can be defined as sets of (codes for) ordered pairs of field elements.  Polynomials can be encoded by finite
strings of coefficients, so polynomial rings are sets of (codes for) finite strings, with related ring operations.
For details pertaining to any of these definitions, see either of the references above.

Our study of fields begins in the next section with the definition of an algebraic field extension.  To simplify the exposition in
sections \S\ref{sectionext} through \S\ref{sectioniso}, we restrict our discussion to characteristic 0 fields.
Consequently, in these sections all irreducible polynomials are separable.  We indicate how to extend
results of earlier sections to fields of other characteristics in section \S\ref{sectionotherchar}.

\section {Algebraic extensions and algebraic closures}\label{sectionext}
We provide a definition of algebraic field extension in the context of second order arithmetic and give a few examples of fields and extensions which $\rca$ proves exist.
Our definition of an algebraic extension extends the
definition of algebraic closure in Simpson's book \cite{simpson}*{Definition II.9.2}.
The definition uses the following notational shorthand.
Given a field $F$, $a\in F$, $f(x)= \sum_{i \in I} c_i x^i$ a polynomial in $F[x]$, and $\varphi$ a field embedding of $F$,
we write $\varphi (f) = \sum_{i \in I} \varphi(c_i)x^i$ and $\varphi (f)(a) = \sum_{i \in I} \varphi(c_i)a^i$.

\begin{defn} $(\rca )$\label{defnae}\label{defnac}
An {\sl algebraic extension} of a countable field $F$ is a pair $\langle K , \varphi \rangle$ where
$K$ is a countable field,
$\varphi$ is an embedding of $F$ into $K$, and for every $a \in K$ there is a nonzero $f(x) \in F[x]$ such that $\varphi(f)(a) = 0$.
When appropriate, we drop the mention of $\varphi$ and denote the extension by $K$ alone.

If $K$ is an algebraic extension of $F$ that is algebraically closed, we say $K$ is an {\sl algebraic closure} of $F$, and often write $\clo F$ for $K$.
\end{defn}

$\rca$ can prove the existence of algebraic closures, as shown in Theorem 2.5 of Friedman, Simpson, and Smith \cite{fss}.
However, the notation $\clo F$ in the preceding definition is somewhat misleading, since $\rca$ does not
prove the uniqueness of algebraic closures up to isomorphism.  To be specific, Theorem 3.3 of Friedman, Simpson, and Smith \cite{fss}
shows  that the statement ``for every field $F$, the algebraic closure of $F$ is unique up to isomorphism'' is equivalent
to $\wkl$.
As for other algebraic extensions, we often drop $\varphi$ and simply
denote an algebraic closure by $\clo F$.

In order to describe the images of fields under embeddings, Friedman, Simpson, and Smith \cite{fss} introduce
the notion of a $\Sigma^0_1$-subfield.

\begin{defn}$(\rca )$\label{defnssf}
Suppose $K$ is a countable field.  A $\Sigma^0_1$ formula $\theta (x)$
defines a $\Sigma^0_1${\sl-subfield} of $K$ if
\begin{enumerate}
\item  $\forall x ( \theta (x) \to x \in K )$,
\item  $\theta ( 0_K)$ and $\theta (1_K )$ (where $0_K$ and $1_K$ are the additive and multiplicative identities of $K$), and
\item  $\forall x \forall y ((\theta (x) \land \theta (y ) ) \to
(\theta (x+y) \land \theta (x-y ) \land \theta (x\cdot y )\land \theta(x / y ) ))$.
\end{enumerate}
Additionally, if $\langle K , \varphi \rangle$ is an algebraic extension of $F$ and $ \theta (\varphi (x))$ for all $x \in F$, we say $\theta (x)$ defines a
$\Sigma^0_1$-$F$-subfield of $K$.
\end{defn}

As noted by Friedman, Simpson, and Smith \cite{fss}, $\rca$ proves that every $\Sigma^0_1$-subfield is the isomorphic image of some field.  The
following transport of structure lemma shows that every $\Sigma^0_1$-$F$-subfield of an algebraic extension of $F$ is the isomorphic image of some algebraic extension
of $F$.  The lemma simplifies the construction of a wide variety of useful field extensions in $\rca$.

\begin{lemma}$(\rca )$\label{sigma}
If $\langle K , \varphi \rangle$ is an algebraic extension of $F$ and $\theta (x)$ defines  a $\Sigma^0_1$-$F$-subfield of $K$, then there is an algebraic
extension $\langle G , \psi \rangle$ of $F$ and an embedding $\tau$ of $G$ into $K$ such that
\begin{enumerate}
\item  $(\forall x \in F) ( \varphi (x) = \tau ( \psi (x)))$, and
\item  $(\forall x \in K )(\theta (x) \leftrightarrow (\exists y \in G)( \tau (y) = x))$.
\end{enumerate}
\end{lemma}

\begin{proof}
If the subfield defined by $\theta$ is finite, then the theorem is trivial.
Let $\langle K , \varphi \rangle$ be an algebraic extension of $F$ and suppose $\theta$ defines an infinite $\Sigma^0_1$-$F$-subfield of $K$.  Since $\theta$
is a $\Sigma^0_1$ formula, $\rca$ proves the existence of an injective function $\tau : \nat \to K$ that enumerates all those elements of $K$ for which $\theta$
holds.  Without loss of generality, we may assume that $\tau ( 0 ) = 0_K$ and $\tau (1) = 1_K$.  Define field operations $+$ and $\cdot$ on $\nat$ by
$i+j = \tau^{-1}( \tau (i) + \tau (j))$ and
$i\cdot j =\tau^{-1}(\tau (i) \cdot \tau (j))$.
Let $G$ denote $\nat$ with these operations.  Define $\psi : F \to G$ by letting
$\psi (x) = \tau^{-1} ( \varphi (x) )$ for each $x \in F$.  Since $\theta$ defines a $\Sigma^0_1$-$F$-subfield of $K$, $\rca$ proves that $\psi$ and the
field operations of $G$ all exist and are all total.  Routine verifications show that $\langle G , \psi \rangle$ and $\tau$ satisfy the conclusions of the theorem.
\end{proof}

In later constructions, it is convenient to have ready access to familiar field extensions of $\rat$.
Working in $\rca$, we can fix a representation of $\rat$, for example that in Theorem II.4.2 of Simpson \cite{simpson}.
By Theorem 2.5 of Friedman, Simpson, and Smith \cite{fss}, we can find $\clo \rat$, an algebraic closure of $\rat$.
As a concrete example of a specific extension, we can locate the first element of $\clo \rat$ satisfying $x^2 -2 = 0$,
and denote it by $\sqrt{2}$.
The collection of terms of the form $q_0 + q_1 \sqrt{2}$ with $q_0  , q_1 \in \rat$
is a $\Sigma^0_1$-subfield of $\clo \rat$.  By Lemma \ref{sigma}, $\rca$ proves that there is an algebraic extension
of $\rat$ that is isomorphic to this $\Sigma^0_1$-subfield;  we denote it by $\rat(\sqrt{2} )$.
In the minimal model of $\rca$ consisting of $\omega$ and the computable sets, this field is a computable presentation
of $\rat ( \sqrt{2} )$; in this case, an algebraist might say it {\emph {is}} $\rat (\sqrt{2})$.
Similarly, for any sequence $\langle \alpha_i \mid i \in \nat \rangle$
of elements in $\clo \rat$, $\rca$ proves the existence of the algebraic extension $\rat( \alpha_i \mid i \in \nat )$.
If we like, we can apply Theorem 2.12 of \cite{fss}, take the algebraic closure of the real closure of $\rat$, and
adjoin a real (or non-real) cube root of $2$ to $\rat$ in the same fashion.
Similar constructions can be carried out over other base fields.

Besides proving the existence of all these field extensions, $\rca$ can prove many useful results about them.  The following two examples play an important
role in the next section.

\begin{lemma}\label{roththm}
$(\rca )$
Let $p _1 , \dots , p _n$ and $q_1 , \dots , q_r$ be disjoint lists of distinct primes.  Then $\sqrt {q_1 \dots q_r} \notin \rat (\sqrt {p_1} , \dots , \sqrt {p_n} )$.
\end{lemma}

\begin{proof}
This is a formalization of the main theorem in the paper of Roth \cite{roth}.  His argument is essentially an application of $\Pi^0_1$-$\sf{IND}$, which is
provable in $\rca$ by Corollary II.3.10 of Simpson \cite{simpson}.  For a sketch of a generalization of this result to fields other than $\rat$, see Lemma \ref{rothalt} in section \S\ref{sectionotherchar}.
\end{proof}

\begin{lemma}\label{rothcoro}
$(\rca )$
Let $p _1 , \dots , p _n$ and $q_1 , \dots , q_r$ be disjoint lists of distinct primes.  Then $\sqrt {q_1} \notin \rat (\sqrt {p_1} , \dots , \sqrt {p_n}, \sqrt {q_1 q_2} ,\dots ,\sqrt{q_1 q_r} )$.
\end{lemma}

\begin{proof}
Suppose $p _1 , \dots , p _n$ and $q_1 , \dots , q_r$ are as specified and the lemma fails.  Write $\sqrt {q_1 }$ as a linear combination of
products of elements of $\{ \sqrt {p_1} , \dots , \sqrt {p_n}, \sqrt {q_1 q_2} ,\dots ,\sqrt{q_1 q_r} \}$ with coefficients in $\rat$.  Separating the summands in which $\sqrt {q_1}$ appears an
even number of times from those in which it appears an odd number of times, we may write $\sqrt {q_1} = \alpha + \beta \sqrt{q_1}$ where $\alpha$ and $\beta$
are elements of $F = \rat (\sqrt{p _1} , \dots , \sqrt{p _n}, \sqrt{q_2} , \dots , \sqrt{q_r })$ and $\beta$ contains some $\sqrt {q_i }$ for $2 \le i \le r$.
Since $\beta \neq 1$ implies $\sqrt{q _1 } \in F$,
contradicting Lemma \ref{roththm}, we must have $\beta = 1$.  Since $\beta$ contains some $\sqrt{q_i }$ for $2 \le i \le r$, we can separate and solve for
$\sqrt{q_i }$, showing that $\sqrt{q_i } \in \rat ( \sqrt {p_1} , \dots , \sqrt {p_n} , \sqrt{q_2 } , \dots , \sqrt{q_{i-1} } , \sqrt {q_{i+1}} , \dots , \sqrt {q_r} )$, again
contradicting Lemma \ref{roththm}.  Thus, the lemma must hold.
\end{proof}

\section{Extensions of isomorphisms}\label{sectionauto}

We analyze the strength required to extend an isomorphism between two fields to an isomorphism between their algebraic closures.  If $K$ and $J$ are isomorphic fields, then the isomorphism extends to an isomorphism of $\clo K$ and $\clo J$.  This type of extension can be
used to show that if $F$ is not algebraically closed, then there is an automorphism of $\clo F$ that fixes $F$ but is not the identity.  As $F$ is not algebraically closed, there is a irreducible polynomial in $F[x]$ with distinct roots $\alpha$ and $\beta$ in $\clo F$.  The fields $F(\alpha)$ and $F(\beta)$ are isomorphic by an isomorphism that fixes $F$ and sends $\alpha$ to $\beta$, and this isomorphism extends to an automorphism of $\clo F$ that fixes $F$ but is not the identity.  We show that in general $\wkl$ is required to extend an isomorphism between two fields to their algebraic closures and to produce a nonidentity automorphism of $\clo F$ that fixes $F$ when $F$ is not algebraically closed.

\begin{defn}\label{defniso}
$(\rca )$
Suppose $\langle K, \varphi \rangle$ and $\langle J, \psi \rangle$ are
algebraic extensions of $F$.  We say $K$
{\sl is embeddable in $J$ over F}
(and write $K \preceq _F J$) if there is an embedding $\tau : K \to J$ such that for all
$x \in F$, $\tau (\varphi (x) ) = \psi (x)$.  We also say that $\tau$ {\sl fixes} $F$ and
call $\tau$ an $F$-embedding.
If $\tau$ is also bijective, we say $K$
{\sl is isomorphic to $J$ over $F$}, write $K \cong _F J$, and call $\tau$ an $F$-isomorphism.
\end{defn}

Informally, when $\langle K, \varphi \rangle$ and $\langle J, \psi \rangle$ are algebraic extensions of $F$, one identifies
$F$ both with its image in $K$ under $\varphi$ and also with its image  in $J$ under $\psi$.
Given such identifications, if $\tau$ fixes $F$ as in the
preceding definition, then $F$ is in the domain of $\tau$ and for all $x \in F$, $\tau(x)=x$.
In the formal setting, the preceding definition describes the relationship between $K$ and $J$
without asserting that $F$ is a subset of $K$ or $J$.  Similarly in the following definition,
the phrases ``$\theta$ extends $\tau$'' and ``$\theta$ restricts to $\tau$'' do not imply that $F$ is a subset of $K$ or that $G$ is
a subset of $H$.

\begin{defn}\label{defnext}
$(\rca )$
Suppose $\tau : F \to G$ is a field embedding, $\langle K, \varphi \rangle$ is an extension of $F$,
$\langle H, \psi \rangle$ is an extension of $G$, and $\theta :K \to H$ satisfies
$\theta (\varphi (v)) = \psi (\tau (v))$ for all $v \in F$.  Then we say $\theta$ {\sl extends}
$\tau$, $\theta$ {\sl is an  extension of} $\tau$, $\theta$ {\sl restricts to} $\tau$, and $\tau$ {\sl is a restriction of} $\theta$.
\end{defn}

Using the preceding definitions, we can formalize the following version of Theorem 1.8
of Hungerford \cite{hungerford} and prove it in $\rca$.

\begin{thm}\label{hungerford18}
$( \rca )$
If $\tau$ is an isomorphism from a field $F$ onto a field $G$ and
$\alpha \in \clo F$ is a root of an irreducible polynomial $p(x)$ of $F[x]$, then for any
root $\beta$ of $\tau (p)(x)$ in $\clo G$, there is an isomorphism of $F(\alpha )$
onto $G( \beta )$ which extends $\tau$.  In particular, taking $F = G$, we have that
if $p(x)$ is
an irreducible polynomial over $F$ with roots $\alpha$ and $\beta$ then
$F( \alpha ) \cong_F F(\beta )$.
\end{thm}

\begin{proof}
Suppose $F$, $G$, $\tau$, $p$, $\alpha$, $\beta$
are as in the hypothesis of the theorem, and let $\langle F(\alpha ) , \varphi \rangle$ and $\langle G(\beta ) , \psi \rangle$
be the associated algebraic extensions.
In order to define a map $\theta: F(\alpha ) \to G( \beta )$ extending $\tau$, we need to characterize a typical element of
$F(\alpha )$.  Recall that $F(\alpha )$ is isomorphic to a $\Sigma^0_1$-$F$-subfield of $\clo F$ containing $\alpha$, so
let $\alpha_e \in F(\alpha )$ be the pre-image of $\alpha$ under this isomorphism.  Define $\beta_e \in G(\beta )$ similarly.
Then for every element $\gamma \in F(\alpha )$ we can uniformly find polynomials $q(x)$ and $r(x)$ in $F[x]$ such
that $\gamma = \frac {\varphi(q) (\alpha_e )}{\varphi (r)(\alpha _ e )}$.  For any such $\gamma$, define
$\theta (\gamma ) = \frac {\psi(\tau(q))(\beta_e ) }{\psi (\tau(r))(\beta _ e)}$.  Using the fact that $p(x)$ is
irreducible over $F$, one can prove that if $\varphi (r ) ( \alpha_e ) \neq 0$, then $\psi (\tau (r)) ( \beta _e ) \neq 0$.
Thus $\theta (x)$ is well-defined.  The subset of $F(\alpha ) \times G(\beta)$ defining $\theta$ exists by
$\Delta^0_1$ comprehension.  Verification of the remaining properties of $\theta$ can be
proved without further uses of comprehension or induction.  In particular, the proof that $\theta$ is single-valued relies on
the fact that $p(x)$ is irreducible over $F$.
The proofs that $\theta$ preserves operations and is onto $G(\beta )$
rely on the fact that $\tau$ is an isomorphism of $F$ onto $G$.
Given that $\theta$ is single-valued and that the isomorphisms
map multiplicative identities to multiplicative identities, one can prove that $\theta$ extends $\tau$.
\end{proof}

Ordinarily, one can iterate Hungerford's theorem to create automorphisms of algebraic closures.  Proving the existence of such
extensions inherently demands greater logical strength than Hungerford's theorem alone, as shown by the following result.  Other
results related to iteration of Hungerford's theorem appear as Theorems \ref{tower1} and \ref{tower2} in section \S\ref{sectioniso}.

\begin{thm}\label{diamond}
$ ( \rca )$
The following are equivalent:
\begin{enumerate}
\item  $\wkl$.
\item  \label{diamond2} Let $F$ be a field with algebraic extensions $K$ and $ K^\prime$.
If $\varphi$ is an isomorphism witnessing
$K \cong_F K^\prime$, then $\varphi$ extends to an isomorphism witnessing $\clo K \cong_F \clo {K^\prime}$.
In the case when $\clo K = \clo {K^\prime}$, $\varphi$ extends to an $F$-automorphism of $\clo K$.
\item \label{diamond3} Let $F$ be a field with an algebraic closure $\clo F$.  If $\alpha \in \clo F$ and
$\varphi : F(\alpha ) \to F(\alpha)$ is an $F$-automorphism of $F(\alpha )$,
then $\varphi$ extends to an $F$-automorphism of $\clo F$.
\end{enumerate}
Furthermore, if $K$ is a subset of $\clo K$ fixed by its embedding, then (2) is provable in $\rca$.  Similarly, if $F(\alpha )$ is a subset of $\clo F$ fixed by its embedding,
then (3) is provable in $\rca$.
\end{thm}

\begin{proof}
We will work in $\rca$ throughout.  To prove that (1) implies (2), assume $\wkl$ and let
$F$, $K$, $K^\prime$, and $\varphi$ be as in the hypothesis of (2).  Let $\langle \clo K , \tau \rangle$
and $\langle \clo {K^\prime} , \tau^\prime \rangle$ be algebraic closures of $K$ and $K^\prime$.
Then $\langle \clo {K^\prime }, \tau^\prime \circ \varphi \rangle$ is an algebraic closure of $K$.
By Theorem 3.3 of Friedman, Simpson, and Smith \cite{fss}, $\wkl$ implies the uniqueness
of algebraic closures.  (This theorem also appears as Lemma IV.5.1 in Simpson \cite{simpson}
in a formulation that serves our purposes particularly well.)  Thus there is an isomorphism
$\psi : \clo K \to \clo{K ^\prime}$ such that for all $x \in K$, $\psi (\tau (x)) = \tau^\prime ( \varphi (x))$.  By
Definition \ref{defnext}, $\psi$ extends $\varphi$.  Since $\varphi$ fixes $F$, so does $\psi$.  Thus $\psi$
witnesses $\clo K \cong_F \clo{K^\prime}$.

Since (\ref{diamond3}) is a restriction of (\ref{diamond2}), we can complete the proof of the theorem by showing that (\ref{diamond3}) implies
$\wkl$.  It suffices to use (\ref{diamond3}) to separate the ranges of two injections with no common values.
Let $f$ and $g$ be injections such that for all $i$ and $j$, $f(i) \neq g(j)$.  Without loss of generality,
we may assume that $0$ is not in the range of either function.  Let $p_i$ denote the
$i^{\text{th}}$ prime, where $2$ is the $0^{\text{th}}$ prime.  By Lemma \ref{sigma} the field
$F = \rat ( \sqrt{\mathstrut p_{f(i)}} , \sqrt{\mathstrut 2 p_{g(i)}} \mid i \in \nat )$ exists.
By Lemma \ref{rothcoro}, $\sqrt {\mathstrut 2} \notin F$.
On the other hand, we may chose $\clo F =\clo \rat$, so
$\sqrt{\mathstrut 2} \in \clo F$.
Define $\varphi$ on $F(\sqrt{\mathstrut 2})$
by $\varphi ( a+ b \sqrt{\mathstrut 2} ) = a-b \sqrt{\mathstrut 2} $.  Note that every value of $F(\sqrt{\mathstrut 2} )$ can
be written uniquely in the form $a+ b \sqrt{\mathstrut 2} $.  By (\ref{diamond3}), $\varphi$ can be extended to an automorphism
$\clo \varphi$ of $\clo F$ that fixes $F$.  By recursive comprehension, the set
$S = \{ i \mid \clo\varphi ( \sqrt {\mathstrut p_i} ) = \sqrt {\mathstrut p_i} \}$ exists.  For any $i$,
$\sqrt{\mathstrut p_{f(i)}} \in F$, so $f(i) \in S$.  Also, $\sqrt{\mathstrut 2 p_{g(i)}} \in F$, so
$\clo \varphi ( \sqrt{\mathstrut 2 p_{g(i)}}) = \sqrt{\mathstrut 2p_{g(i)}} = \sqrt{\mathstrut 2}\sqrt{ \mathstrut p_{g(i)}}$.
Since $\clo\varphi$ is a homomorphism,
$\clo\varphi ( \sqrt{\mathstrut 2 p_{g(i)}}) = \clo\varphi ( \sqrt{\mathstrut 2})\clo\varphi (\sqrt{ \mathstrut p_{g(i)}})=-\sqrt{\mathstrut 2}\clo\varphi (\sqrt{ \mathstrut p_{g(i)}})$.
Thus $\clo\varphi ( \sqrt {\mathstrut p_{g(i)}})= - \sqrt{\mathstrut p_{g(i)}}$, so $g(i) \notin S$.  Thus $S$ is the desired
separating set.  This completes the proof of the equivalence results.

To prove the final two sentences of the theorem, consider item (2) and suppose $K$ is a subset of $\clo K$.
By Lemma 2.7 and Lemma 2.8
of Friedman, Simpson, and Smith \cite{fss}, given any finite extension of $K$, we can uniformly find all the irreducible
polynomials of the extension.  In particular, we can locate the first such polynomial in some enumeration
of all the polynomials in $\clo K [x]$.  Let $\seq {p}{i}$ and $\langle {{\langle r_{ij} \rangle _{j \le j_i}}} \rangle_{i \in \nat}$ be sequences such that
for each $i$, $p_i$ is the first irreducible polynomial of $K ( r_{tj} \mid t<i \land j \le j_t )[x]$, and $\langle r_{ij} \rangle_{j \le j_i }$
are the roots of $p_i$ in $\clo K$.  Let $\langle r_{0j}^\prime \rangle_{j\le j_0}$ be the roots of $\varphi (p_0 )$.
Any $k$ in $K(r_{0j} \mid j \le j_0 )$ can be written as $q( r_{00}, \dots , r_{0j_0} )$ for some $q \in K[x_0 , \dots , x_{j_0}]$.  Define
$\varphi^*(k) = \varphi (q)(r_{00}^\prime , \dots , r_{0j_0}^\prime )$.  In general, if $\varphi^*$ is defined on
$K (r_{tj} \mid t<i \land j \le j_t )$, let
$\langle r_{ij}^\prime \rangle_{j \le j_i }$ be the roots of $\varphi^* (p_i )$ and for
$k \in K (r_{tj} \mid t<i \land j \le j_t )(r_{ij} \mid j \le j_i )$, let
$\varphi (k) = \varphi^* (q) (r_{i0}^\prime , \dots , r_{i j_i} ^\prime )$.
Routine arguments verify that $\varphi^*$ witnesses $\clo K \cong_F \clo{K^\prime}$ and extends $\varphi$.  As noted before, item (3)
is a special case of item (2), so $\rca$ also suffices to prove (3) when $F(\alpha ) \subset \clo F$.
\end{proof}

In section 5 of their paper \cite{mn}, Metakides and Nerode
construct a computably presented field $F$ in an extension $K$ such that the only computable $F$-automorphism
of $K$ is the identity.  Their proof gradually constructs $F$ while diagonalizing to avoid computable nontrivial automorphisms.  The reversal
of the following theorem may be viewed as the construction of a computably presented field such that every nontrivial $F$-automorphism
of $\clo F$ encodes a separating set for computably inseparable computably enumerable sets.

\begin{thm}\label{mn}
$ ( \rca )$
The following are equivalent:
\begin{enumerate}
\item  $\wkl$.
\item  \label{mn2} Let $K$ be a proper algebraic extension of $F$ and let $\clo K$ be
an algebraic closure of $K$.  Then
there are at least two  $F$-embeddings of $K$ into $\clo K$.
\item  \label{mn3p}  Let $\langle K , \psi \rangle$ be an algebraic extension of $F$. Suppose that
every irreducible polynomial over $F$ that has a root in $K$ splits into linear factors in $K$.
(This is called $\nor{1}$ in Definition \ref{defnnormal}.)   If $\alpha \in K$  and $\alpha$ is not in the range of $\psi$, then there
is an $F$-automorphism $\varphi$ of $K$ such that $\varphi(\alpha) \neq \alpha$.
\item  \label{mn3} If $F$ is not algebraically closed, then there is an $F$-automorphism
of $\clo F$ that is not the identity.
\end{enumerate}
\end{thm}

\begin{proof}
To see that (1) implies (\ref{mn2}), assume $\wkl$ and let $\langle K , \psi \rangle$
be an algebraic extension of $F$ and let $\langle \clo K , \tau \rangle$ be an algebraic closure of $K$.
Let
$\alpha$ be an element of $K$ that is not in the range of $\psi$.
By the separability of $F$, the minimal polynomial of $\alpha$ in $F[x]$ has a root $\beta\in \clo K$ such that $\tau(\alpha) \neq \beta$.
By Theorem \ref{hungerford18}, there is an isomorphism $\varphi$ of $F(\alpha )$ onto
$F(\beta )$.  Using $\wkl$, we can apply item (\ref{diamond2}) of Theorem \ref{diamond} and extend $\varphi$
to an $F$-automorphism of $\clo K$.  Restricting this extended map to $K$ yields an $F$-embedding of $K$ into
$\clo K$ which is distinct from $\tau$.

Since $F$-embeddings must map any roots of a polynomial over $F$ to roots of the same polynomial,
adding the splitting hypothesis to (\ref{mn3p}) insures that the $F$-embedding of (\ref{mn2}) is also an automorphism on
$K$.  Thus (\ref{mn2}) implies (\ref{mn3p}).  Since $\clo F$ satisfies the splitting hypothesis of (\ref{mn3p}) and the
automorphism of  (\ref{mn3p}) is not the identity, (\ref{mn3p}) implies
(\ref{mn3}).  It remains only to show that (\ref{mn3}) implies $\wkl$.

As in the proof of the reversal of Theorem \ref{diamond}, it suffices to use (\ref{mn3}) to separate the ranges of
injections $f$ and $g$ satisfying $0 \neq f(s) \neq g(t)\neq 0$ for all $s$ and $t$.  As a notational convenience, we identify the ordered
pair $(i,j)$ with its integer code $(i+j)^2 + i$.  (This coding of pairs is described in Section II.2 of Simpson's book \cite{simpson}.)
Enumerate the polynomials in $\Qb [x]$, with $x^2 - 2$ occurring first in the ordering.
Because we will be working with finite extensions of $\Qb$, Lemma 2.8 of
Friedman, Simpson, and Smith \cite{fss} shows that $\rca$ suffices to determine which polynomials are
irreducible over any of these extensions.  Their Lemma 2.6 \cite{fss} proves the existence of primitive elements in $\rca$.
Define sequences
$\seq{v}{i}$ of algebraic numbers and $\seq{d}{i}$ of degrees of polynomials as follows.
If $i = (j,0)$ for some $j$, let $g(x)$ be the next irreducible polynomial which does not split into linear factors
over $\Qb (v_k \mid k < i )$.  Let $G$ be the splitting field of $g(x)$ over $\Qb ( v_k \mid k<i )$.  Let $v_i$ be a primitive
element for $G$ over $\Qb (v_k \mid k<i )$, and let $d_i$ be the degree of $v_i$ over
$\Qb (v_k \mid k < i )$.  Since $x^2-2$ is the first polynomial and $(0,0) = 0$, $v_0 = \sqrt{2}$
(or some other primitive element for $\Qb (\sqrt {2} )$) and $d_i = 2$.  If $i = (j,n)$ and $n>0$, let
$d_j$ be the degree of $v_j$ over $\Qb$ and let $p$ be the first prime such that $x^{d_j} - p$ is irreducible over
$\Qb ( v_k \mid k < i )$.  Let $v_i = \frac{p^{1/{d_j}}}{v_j}$ and let $d_i$ be the degree of $v_i$ over $\Qb (v_k \mid k <i )$.
Note that the degree of $v_j v_i$ over $\Qb (v_k \mid k<i )$ is $d_j$ and $d_j \le d_i$.
By Lemma 2.6 and Lemma 2.8 of
Friedman, Simpson, and Smith \cite{fss}, the
sequences $\seq {v}{i}$ and $\seq{d}{i}$ can be constructed in $\rca$.
By our construction, for each $i$ the set of products
$\{ \prod_{j<i} v_j^{e_j} \mid \forall j ( 0 \le e_j < d_j )\}$ is a vector space basis
for $\Qb (v_k \mid k<i )$ over $\Qb$.  Also, $\{ 1 , v_i , \dots , v_i ^{d_i - 1} \}$ is a basis
for $\Qb (v_k \mid k \le i )$ over $\Qb ( v_k \mid k<i )$.  These claims can be proved in $\rca$
by imitating the proof of Proposition 1.2 in Lang \cite{lang}.

In order to apply (\ref{mn3}), use 
Lemma \ref{sigma} and let
$F= \Qb ( v_{(i,f(j))} , v_{(i,g(j))} \cdot v_i \mid i,j \in \Nb )$ and $\clo F = \clo \rat$.
Assume for a moment that $\bQb$ is a nontrivial extension; details are given below.
Applying (\ref{mn3}), there is a nontrivial $F$-automorphism $\varphi$ of $\bQb$.
If $\varphi$ fixed every $v_i$, then $\varphi$ would be the identity on $\bQb$, so we can
fix some $i$ such that $\varphi(v_i ) \neq v_i $.  Since $\varphi$ fixes $F$, for every $j \in \Nb$,
$\varphi ( v_{(i,f(j))})= v_{(i,f(j))}$, and
$\varphi (v_{(i,g(j))} \cdot v_i ) = v_{(i,g(j))} \cdot v_i $.
Since $\varphi(v_{(i,g(j))} ) = v_{(i,g(j))}$ implies $\varphi(v_i ) = v_i$, we must
have $\varphi (v_{(i,g(j))}) \neq v_{(i,g(j))}$.  By $\Delta^0_1$ comprehension,
the separating set $\{k \mid \varphi (v_{(i,k)})= v_k \}$ exists.

To complete the proof of the reversal and the proof of the theorem, it remains only to show
that the field $F$ defined above is a proper subfield of $\bQb$.  Suppose by way of contradiction
that $\sqrt {2} \in F$.  Since $F$ is generated by elements of the bases we constructed, we may write
$\sqrt{2}$ as a linear combination of products of generators of $F$.  We will use $j_0$ and $j_1$ to denote
components of the pair encoded by $j$, so $j = (j_0 , j_1 )$.  Let
\[
\sqrt{2} = \sum_{i \in I } q_i \prod _{j \in J_i } ( v_{j_0} v_{(j_0 , g(j_1)) })^{e_j} \prod _{k \in K_i } v_{(k_0 , f(k_1)) }^{e_k}
\]
where $I$, $J_i$, and $K_i$ denote finite sets of integers, $0 < e_j < d_{j_0}$, and $0 < e_k < d_{(k_0 , f(k_1 ) ) }$.
For a sufficiently large value of $i$, all the products on the right are elements of the basis
$B_i = \{ \prod_{j<i} v_j^{e_j} \mid \forall j ( 0\le e_j < d_j )\}$ for $\Qb (v_k \mid k<i )$ over $\Qb$, as is $v_0 =\sqrt{2}$.
By linear independence of $B_i$, there must be some $i_0$ and some $q \in \Qb$ such that:
\[
\sqrt{2} = q \prod _{j \in J_{i_0} } ( v_{j_0} v_{(j_0 , g(j_1)) })^{e_j} \prod _{k \in K_{i_0}} v_{(k_0 , f(k_1)) }^{e_k} 
\]
Let $s$ be the largest subscript appearing on a $v$ in this product.  Since $g$ is nonzero, $g(j_1)>0$, so by
the definition of the pairing function we
have $j_0 < (j_0 , g(j_1))$.  Thus $s$ is of the form $(j_0 , g(j_1 ))$ or $(k_0 , f(k_1 ) )$.
Since the ranges of $f$ and $g$ are disjoint, only one of these may hold.  Thus for
some $0 < e < d_s$, $v_s^e \in \Qb (v_i \mid i<s )$, contradicting our construction of $F$.
This shows that $\sqrt{2} \notin F$ and completes the proof.
\end{proof}

As noted before the presentation of the preceding theorem, it has an immediate corollary in computable field theory.

\begin{cor}
Given any pair of disjoint computably enumerable sets,
there is a computable field $F$ that is not algebraically closed and has
a computable algebraic closure $\clo F$
such that any nontrivial $F$-automorphism of $\clo F$ computes a separating
set for the computably enumerable sets.  In particular, if the computably enumerable sets are computably
inseparable, then any nontrivial
$F$-automorphism is noncomputable.  Additionally,
every computable field $F$ that is not algebraically closed has a computable
algebraic closure $\clo F$, and any such closure has a nontrivial $F$-automorphism $\varphi$
such that $\varphi ^\prime \le_T {\bf 0}^\prime$.
\end{cor}

\begin{proof}
To prove the first part of the corollary, imitate the construction from
Theorem \ref{mn}, using computable enumerations of the disjoint
c.e.~sets as the functions with disjoint ranges.  To prove the last sentence,
note that Theorem VIII.2.17 of \cite{simpson} proves the
existence of a model of $\wkl$ consisting of only low sets.
This model contains all the computable fields, an
algebraic closure of each one, and by Theorem \ref{mn}, the
desired nontrivial automorphism.  One could avoid the discussion of models by applying
the Jockusch/Soare low basis theorem, Theorem 2.1 of \cite{js}, to a computably
bounded computable tree constructed as in the proof of Theorem \ref{diamond}.
\end{proof}

The constructions of this section can be used to find computable
binary trees whose infinite
paths can be matched in a degree preserving fashion with the $F$-automorphisms of $K$ for
appropriately chosen fields $F$ and $K$.  Since the degree of $K$ over $F$ is either finite
or countable, the number of $F$-automorphisms of $K$ is either finite or the continuum.
Many computable binary trees have countably many infinite paths.  Thus, given an arbitrary
computable binary tree, we cannot expect to be able to construct fields so that the automorphisms
match the infinite paths.  This is reminiscent of the argument for why Remmel's result on $3$-colorings of
graphs \cite{remmel} does not extend to $2$-colorings.  It would be nice to know if some analog of Remmel's
result holds in an algebraic setting.

\begin{question}
Is there a nice characterization of those computable binary trees whose infinite paths can be matched
via a degree preserving bijection to the $F$-automorphisms of $K$ for some computable extension $K$
of a computable field $F$?  How does this class of trees compare with similar classes for automorphisms
of other computable algebraic structures?
\end{question}

\section{Extensions of embeddings}\label{sectioniso}

Informally, if $J$ and $K$ are algebraic extensions of $F$,
and both $F(j) \preceq_F K$ for every $j\in J$
and $F(k) \preceq_F J$ for every $k \in K$, then
$J$ is $F$-isomorphic to $K$.  The proof that $J \cong_F K$ can
be carried out in two steps:  First prove that $J \preceq_F K$ and
$K \preceq_F J$ and second deduce the existence of the isomorphism.
This second step can be carried out in $\rca$.

\begin{thm}\label{antisym}
$(\rca )$  If $ J \preceq_F K$ and $K \preceq_F J$, then $J \cong_F K$.
\end{thm}

\begin{proof}
Suppose $\langle J, \psi \rangle$ and $\langle K , \varphi \rangle$ are algebraic extensions
of $F$, $\theta: J \to K$ embeds $J$ into $K$, and $\tau: K \to J$ embeds $K$ into $J$.
We need only show that $\theta$ is onto.  Fix $k_0 \in K$.  Let $p \in F[x]$ be
the minimal polynomial for $k_0$ over $F$ and let $k_0 , \dots , k_n$ be the roots
of $\varphi(p)$ in $K$.  Let $j_0 , \dots , j_m $ be the roots of $\psi (p)$ in $J$.
Since $\theta$ maps $j_0 , \dots , j _m$ one-to-one into $k_0 , \dots, k_n$ and
$\tau$ maps $k_0 , \dots , k_n$ one-to-one into $j_0 , \dots , j_m$, by the
finite pigeonhole principle (which is provable in $\rca$) we must have that $m=n$ and
$k_0$ is in the range of $\theta$.
\end{proof}

In light of Theorem \ref{antisym}, our next goal is to formulate existence theorems
for embeddings.
Of course, 
in any embedding $K \preceq _F J$, each element $k \in K$ must map to a root in $J$ of
its irreducible polynomial.  The next two definitions describe functions that are helpful for
bounding the search for acceptable images of roots.
Eventually, we will prove embedding existence theorems with
bounds (Theorem \ref{tower1}) and without bounds (Theorem \ref{tower2}).

\begin{defn}\label{defnrm}\label{rootmod}
$(\rca )$
Suppose $\langle K , \varphi \rangle$ is an algebraic extension of $F$.  A function $r: F[x] \to K ^{< \nat}$ is a
{\sl root modulus for $K$ over $F$} if for every $p \in F[x]$, $r(p)$ is (a code for) the finite set of all the
roots of $\varphi (p) $ in $K$.
We code finite sets as in Theorem 11.2.5 of
Simpson \cite{simpson}, so the integer code for the set is always greater than the maximum element.
Thus $r(p)$ is also an upper bound on the roots of $\varphi (p)$ in $K$.
\end{defn}

\begin{defn}\label{feb}
$(\rca )$
Suppose $\langle K , \varphi \rangle$ and $\langle J , \psi \rangle$ are algebraic extensions of $F$.  An $F$ {\sl embedding bound}
of $K$ into $J$ is a function $f : K \to J^{< \nat }$ such that for each $k\in K$, $f(k)$ contains all the roots in $J$
of the minimal polynomial of $k$ over $F$.
Equivalently, for $k \in K$
and $j\in J$,
if $\forall p \in F[x]( \psi (p)(j)=0 \to \varphi (p)(k)= 0 )$ then $j \in f(k)$.
By  our choice of coding, $f(k)$ is also an upper bound on the roots in $J$ of the minimal polynomial of $k$ over $F$.
\end{defn}

Suppose $K$ and $J$ are fields, $f$ is an $F$ embedding bound, and $p$ is the minimal polynomial of $k$ over $F$.
Under our definition,
$f(k)$ may contain a finite number of elements that are not roots of $\psi (p)$ in $J$.  Also, $f(k)$ might be empty
if $K$ is not embeddable into $J$.
The next two theorems explore relationships between root moduli and $F$ embedding bounds.
The first theorem shows that a root modulus can act as a sort of universal $F$ embedding bound.

\begin{lemma}\label{rootmodtofeb}
$(\rca)$  Suppose $J$ is an algebraic extension of $F$.  $J$ has a root modulus over $F$ if and only if for every algebraic extension $K$ of $F$, there
is an $F$ embedding bound of $K$ into $J$.  If there is an $F$ embedding bound of $\clo F$ into $J$, then $J$ has a root modulus.
\end{lemma}

\begin{proof}
Suppose $\langle J , \psi \rangle$ is an algebraic extension of $F$.  First, let $r$ be a root modulus for $J$ and let $\langle K ,\varphi \rangle$ be an extension of $F$.
For each $k \in K$, let $p_k$ be the first polynomial in some
enumeration of $F[x]$ such that $\varphi(p_k ) (k) = 0$.  Define $f: K \to J^{< \nat }$ by $f(k) = r(p_k)$.
For $k \in K$, the minimal polynomial of $k$ over $F$ divides $p_k$, so all of its roots are in $f(k)$.  Thus $f$ is an $F$ embedding bound of $K$ into $J$.

Since $\clo F$ is an algebraic extension of $F$, the remaining implication of the second sentence follows from the third sentence.  To prove the third sentence,
suppose $f$ is an $F$ embedding bound of $\clo F$ into $J$.  Given any polynomial $p \in F[x]$, let $q_0 , \dots , q_n$
be a list of all the roots of $p(x)$ in $\clo F$, and define
$r(p) = \{ j \in f(q_0 ) \cup f(q_1) \cup \dots \cup f(q_n)\mid \varphi(p)(j)=0\}$.  $\rca$ proves that $r$ exists and is a root modulus
for $J$.
\end{proof}

General assertions of the
existence of $F$ embedding bounds and root moduli require additional set comprehension.

\begin{thm}
$( \rca )$ The following are equivalent:
\begin{enumerate}
\item   $\aca$.
\item  If $J$ is an algebraic extension of $F$, then $J$ has a root modulus.
\item  If $K$ and $J$ are algebraic extensions of $F$, then there is an $F$ embedding bound of $K$ into $J$.
\end{enumerate}
\end{thm}

\begin{proof}
Working in $\rca$, we begin by proving the equivalence of (1) and (2).
To prove that (1) implies (2), suppose $\langle J, \psi \rangle$ is an algebraic extension of $F$.  Since the finite set of all roots of $\psi (p)$ in $J$ is uniformly
arithmetically definable using $p$ as a parameter, $\aca$ proves the existence of a root modulus for $J$.

To prove that (2) implies (1), let $g: \nat \to \nat$ be an injection.  $\aca$ follows from the existence of the range of $g$.
Let $F = \rat$.  Let $p_i$ denote the $i^{\text{th}}$ prime and
consider $\rat ( \sqrt{p_{g(i)}} \mid i \in \nat )$ as a $\Sigma^0_1$-subfield of some algebraic closure $\clo \rat$ of the rationals.
We can find $\langle J , \psi \rangle $, a field extension of $\rat$, such that $\rat ( \sqrt{p_{g(i)}} \mid i \in \nat )$ is an isomorphic image of $J$ in $\clo \rat$.
Apply (2) to find a root modulus for $J$.  Note that for every natural number $k$,
\[ \exists t  (g(t) = k) \leftrightarrow r(x^2 - p_k ) \neq \emptyset .\]
Since $r (x^2 - p_k )$ is a code for a finite set, $\{ k \mid r(x^2 - p_k ) \neq \emptyset \} $ exists by
$\Delta^0_1$-comprehension.  Thus $\rca$ and (2) suffice to prove the existence of the range of $g$.

Now we turn to the equivalence of (1) and (3).  Since (1) implies (2), by Lemma \ref{rootmodtofeb}, (1) also implies (3).
To prove that (3) implies (1), let $g$, $F=\rat$, and $\langle J , \psi \rangle$ be as in the preceding paragraph.
Let $\langle K , \varphi \rangle$ be a field extension of $F$ such that $\rat ( \sqrt{p_i } \mid i \in \nat )$ is an isomorphic image
of $K$ in $\clo \rat$;  let $\tau$ be that isomorphism.  Apply (3) to find $f:K \to J^{< \nat }$, an $F$ embedding bound  of $K$ into $J$.  Note that
for every natural number $k$,
\[
\exists t (g(t) = k ) \leftrightarrow \exists a (a \in f(\tau^{-1} ( \sqrt{p_k})) \land \psi (a)^2 = p_k ).
\]
Since $f(\tau^{-1} ( \sqrt{p_k}))$ is a finite set, the range of $g$ exists by $\Delta^0_1$-comprehension, completing the proof.
\end{proof}

Despite the fact that root moduli and embedding bounds are not interchangeable, they both can serve to formulate bounded
versions of an embedding theorem.

\begin{thm}\label{tower1}
$(\rca )$  The following are equivalent:
\begin{enumerate}
\item  $\wkl$.
\item  Suppose $K$ and $J$ are algebraic extensions
of $F$ and $f_K$ is an $F$ embedding bound of $K$ into $J$.
If  $F ( k) \preceq _F J$ for all $k \in K$,
then $K \preceq_F J$. 
\item  Suppose $K$ and $J$ are algebraic extensions
of $F$
and
$r_J$ is a root modulus of $J$ over $F$.
If  $F ( k) \preceq _F J$ for all $k \in K$,
then $K \preceq _F J$. 
\end{enumerate}
\end{thm}

\begin{proof}
To prove that $\wkl$ implies (2), let $K$, $J$, $F$, and $f_K$ be as in (2) and suppose
$F(k) \preceq_F J$ for all $k \in K$.  Consider the formula $\theta ( \varphi , k )$ that asserts:
\begin{list}{$\bullet$}{}
\item  $\varphi$ is a subset of $K \times J$.
\item  $\varphi$ preserves field operations.
\item  $\varphi$ is one-to-one.
\item  If $k \in K$, then there is some $j \in f_K (k)$ such that $(k,j) \in \varphi$.
\end{list}
Because $f_K(k)$ is always finite, $\theta (\varphi , k )$ is a $\Pi^0_1$ formula.  For any $n$,
we can find a primitive element $k_0$ for $F( k \mid k \in K \land k < n )$.  Any $\varphi$ witnessing
$F(k_0 ) \preceq_F J$ will also witness $\exists \varphi \forall k < n ~\theta (\varphi , k )$.
By Lemma VIII.2.4.1 of Simpson \cite{simpson}, $\wkl$ proves
$\exists \varphi \forall k ~\theta (\varphi , k )$.  Any $\varphi$ satisfying this formula
$F$-embeds $K$ into $J$.

The proof that (2) implies (3) is immediate from Lemma \ref{rootmodtofeb}.
To prove that (3) implies (1),  note that given two algebraic closures of a field, $\rca$
can prove the existence of the root moduli and embeddings as in (3).  The conclusion of (3) shows
that
each algebraic closure is embeddable in the other.  By Theorem \ref{antisym}, the algebraic closures are $F$-isomorphic.
This implies $\wkl$ by Theorem 3.3 of Friedman, Simpson,
and Smith \cite{fss}.
\end{proof}

The construction used by Miller and Shlapentokh \cite{ms} to prove their Proposition 4.3 can be used
as an interesting alternative proof that (2) implies (1) in the preceding theorem.  The fields in their construction have computable embedding
bounds, but do not have computable root moduli.

In the absence of root moduli and embedding bounds,  the theorem is much stronger.

\begin{thm}\label{tower2}
$(\rca )$  The following are equivalent:
\begin{enumerate}
\item  $\aca$.
\item  Suppose $K$ and $J$ are algebraic extensions
of $F$.
If $F (k ) \preceq _F J$ for all $k \in K$ then $K \preceq _F J$.
\end{enumerate}
\end{thm}

\begin{proof}
To show that $\aca$ implies (2), it suffices to note that given $K$ and $J$ as in (2),  a root moduli for $K$ over $F$ is arithmetically definable.  Since $\aca$ implies $\wkl$, we may apply Theorem \ref{tower1} to find the desired isomorphism.

To prove the converse, let $g \colon \Nb \imp \Nb$ be an injection.  We prove that the range of $g$ exists.  First, extend $\Qb$ to a real closure, then extend the real closure to an algebraic closure $\bQb$.  Since the algebraic closure is a finite separable extension of the real closure, the image of the real closure exists inside the algebraic closure by Friedman, Simpson, and Smith \cite{fss} Lemma 2.6.  This allows us to distinguish the real elements of $\bQb$ from the complex elements of $\bQb$.  Fix an enumeration of $\bQb$, let $\la p_m \ra_{m \in \Nb}$ enumerate the odd primes, and for each $m > 0$, let $\zeta_m \in \bQb$ be the first enumerated primitive $m$\textsuperscript{th} root of unity.  The fields $\Qb(\{2^{1/p_n}, \zeta_{p_n} \mid \exists m (g(m) = n)\})$, $\Qb(\{\zeta_{p_n} \mid \exists m (g(m) = n)\} \cup \{2^{1/p_m} \mid m \in \Nb\})$, and $\Qb(\{\zeta_{p_n} \mid \exists m (g(m) = n)\} \cup \{\zeta_{p_m}2^{1/p_m} \mid m \in \Nb\})$ are all $\Sigma^0_1$-$\Qb$-subfields of $\bQb$.  By Lemma \ref{sigma}, let $F$, $K$, and $J$, be algebraic extensions of $\Qb$ together with embeddings $\tau_F$, $\tau_K$, and $\tau_J$ of $F$, $K$, and $J$, respectively, into $\bQb$ such that
\begin{align*}
\ran(\tau_F) &= \Qb(\{2^{1/p_n}, \zeta_{p_n} \mid \exists m (g(m) = n)\});\\
\ran(\tau_K) &= \Qb(\{\zeta_{p_n} \mid \exists m (g(m) = n)\} \cup \{2^{1/p_m} \mid m \in \Nb\});\\ 
\ran(\tau_J) &= \Qb(\{\zeta_{p_n} \mid \exists m (g(m) = n)\} \cup \{\zeta_{p_m}2^{1/p_m} \mid m \in \Nb\}).
\end{align*}

The field $\Qb(\{2^{1/p_n}, \zeta_{p_n} \mid \exists m (g(m) = n)\})$ is a subfield of both $\Qb(\{\zeta_{p_n} \mid \exists m (g(m) = n)\} \cup \{2^{1/p_m} \mid m \in \Nb\})$ and $\Qb(\{\zeta_{p_n} \mid \exists m (g(m) = n)\} \cup \{\zeta_{p_m}2^{1/p_m} \mid m \in \Nb\})$, so we define maps $\psi_K \colon F \to K$ and $\psi_J \colon F \to J$ by $\psi_K = \tau_K^{-1} \circ \tau_F$ and $\psi_J = \tau_J^{-1} \circ \tau_F$ which witness that $K$ and $J$ are both algebraic extensions of $F$.

To see that $F(k) \preceq_F J$ for all $k \in K$, fix a $k \in K$ and let $M$ be such that $\tau(k) \in \Qb(\{\zeta_{p_n} \mid \exists m (g(m) = n)\} \cup \{2^{1/p_m} \mid m < M\})$.  By  bounded $\Pi^0_1$ comprehension, let $X = \{n < M \mid \neg\exists m(g(m)=n)\}$.  Then $k \in F(\tau_K^{-1}(2^{1/p_n}) \mid n \in X)$, which embeds into $J$ by extending $\psi_J$ so that $\psi_J(2^{1/p_n}) = \zeta_{p_n}2^{1/p_n}$ for each $n \in X$.

By (2), let $\varphi$ be an $F$-embedding of $K$ into $J$.  Let $X$ be the set of numbers $n$ such that $\tau_J(\varphi(\tau^{-1}_K(2^{1/p_n}))) \in \bQb$ is real.  We show that $X$ is the range of $g$.  Suppose $n = g(m)$ for some $m$.  Then $\tau^{-1}_F(2^{1/p_n})$ exists and $\tau^{-1}_K(2^{1/p_n}) = \psi_K(\tau^{-1}_F(2^{1/p_n}))$.  Thus $\varphi(\tau^{-1}_K(2^{1/p_n})) = \varphi(\psi_K(\tau^{-1}_F(2^{1/p_n})))$, and the fact that $\varphi$ is an $F$-embedding means that $\varphi(\psi_K(\tau^{-1}_F(2^{1/p_n}))) = \psi_J(\tau^{-1}_F(2^{1/p_n})) = \tau_J^{-1}(2^{1/p_n})$.  All together, this gives $\tau_J(\varphi(\tau^{-1}_K(2^{1/p_n}))) = \tau_J(\tau_J^{-1}(2^{1/p_n})) = 2^{1/p_n}$, which is real.  On the other hand, if there is no $m$ such that $n = g(m)$, then the only root of $x^{p_n} - 2$ in $\Qb(\{\zeta_{p_n} \mid \exists m (g(m) = n)\} \cup \{\zeta_{p_m}2^{1/p_m} \mid m \in \Nb\})$ is $\zeta_{p_n}2^{1/p_n}$, and $\tau_J(\varphi(\tau^{-1}_K(2^{1/p_n})))$ must be a root of $x^{p_n}-2$.  Thus $\tau_J(\varphi(\tau^{-1}_K(2^{1/p_n}))) = \zeta_{p_n}2^{1/p_n}$, which is not real. 
\end{proof}

\section{Normal extensions and Galois extensions}\label{sectionnorm}

The field theory literature contains a variety of definitions of
normal algebraic extensions.  For example,
Lang \cite{lang} lists three versions corresponding
to $\nor1$, $\nor2$, and $\nor3$ in the following definition.  We add a fourth version to the list that makes use
of the notion of restriction presented in Definition \ref{defnext}.
While algebraists view these as equivalent definitions,
this section shows that the equivalence proofs vary in logical strength.

\begin{defn}\label{defnnormal}
$( \rca )$  Let $\langle K , \psi \rangle$ be an algebraic extension of $F$.
For $1 \le i \le 4$, we say $K$ is a $\nor{ }i${\sl -normal extension of }$F$ if
the condition $\nor{}i$ in the list below holds.
\begin{list}{}{\setlength{\labelwidth}{.2in}}
\item [$\nor{1}$:] If $p(x)\in F[x]$ is irreducible and $\psi (p)(x)$ has a root in $K$, then $\psi (p)(x)$ splits into linear factors in $K$.
\item [$\nor{2}$:] There is a sequence of polynomials over $F$ such that the
image under $\psi$ of each polynomial in the sequence splits into linear factors in $K$, and $K$ is generated by the
roots of these polynomials.  That is, $K$ is the splitting field of the images under $\psi$ of some sequence of polynomials over $F$.
\item [$\nor{3}$:]  If $\varphi : K \to \clo K$ is an $F$-embedding, then $\varphi$ is an $F$-automorphism of $K$.
\item [$\nor{4}$:]  If $\varphi : \clo K \to \clo K$ is an $F$-automorphism, then $\varphi$ restricts to an
$F$-automorphism of $K$.
\end{list}
\end{defn}

Lang \cite{lang} defines Galois extensions as normal separable extensions.  In light of the preceding
list, this yields four reasonable definitions.  Before addressing the equivalence of the various definitions,
we append the following definition from  Hungerford \cite{hungerford}.

\begin{defn}$(\rca)$\label{defngalois}
  A {\sl Galois extension} of the field $F$ is an algebraic extension $K$ of $F$ such that the only elements
  of $K$ that are fixed by all $F$-automorphisms of $K$ are the elements of $F$.  To parallel our $\nor {}i$ notation, we
  will say that Galois extensions have the property $\gal$.
\end{defn}

Usage of the terms ``normal'' and ``Galois'' is far from standardized.
Emil Artin uses ``normal'' for $\gal$ in his Galois Theory \cite{artin}, as does Irving Kaplansky in Fields and Rings \cite{kaplansky}.
Artin and Kaplansky do not use the term ``Galois'' in this sense.
David Hilbert uses ``Galoisscher'' for $\nor{3}$ in Theorie der algebraischen Zahlenk\"{o}rper \cite{hilbert}. Normal doesn't appear in Hilbert's index.
Zariski and Samuel use ``normal'' for $\nor{1}$, pointing out the equivalence with $\nor{2}$, in their Commutative Algebra \cite{zariski}. They only use ``Galois" in the context of finite fields.

\begin{thm}\label{normal1}\label{galoisnormal}
$( \rca )$  For every field $F$ and every algebraic extension $K$ of $F$ we have:
\[\gal \to \nor{1} \leftrightarrow \nor{2} \to \nor {3} \to \nor {4}\]
Moreover, if $F$ is a subset of $K$ fixed by its embedding and $K$ is a subset of $\clo K$ fixed by its embedding, then the four versions of normal are equivalent.
If  the previous conditions hold and $K$ is separable,
then all five conditions are equivalent.
\end{thm}

\begin{proof}
We will work in $\rca$ throughout.  $\nor{1}$ can be deduced from $\nor{2}$ by a straightforward formalization
of the proof of the last theorem in section \S6.5 of Van der Waerden's text \cite{vdw}.  We now turn to the
left to right implications.

 To see that $\gal$ implies $\nor{1}$, let $K$ be a Galois extension of $F$.
  Suppose $p(x)$ is a monic irreducible polynomial over $F$ and that $\psi(p)(x)$ has a root in $K$.
  Let $\alpha_1,\dots,\alpha_k$ be all the roots of $\psi(p)(x)$ in $K$.
  Consider the polynomial $q(x) = (x-\alpha_1)\cdots(x-\alpha_k).$
  Every $F$-automorphism $\varphi$ of $K$ must permute the set $\{\alpha_1,\dots,\alpha_k\}$ and thus the coefficients of $q(x)$ are all fixed by $\varphi$.
  Since $K$ is a Galois extension of $F$, 
  it follows that $q(x) = \psi (r)(x)$ for some $r(x) \in F[x]$.
    Since $r(x)$ divides $p(x)$ and $p(x)$ is monic irreducible, it follows that $p(x) = r(x)$ and hence that $\psi(p)(x)$ (which is $q(x)$)  factors completely in $K$.

To see that $\nor{1}$ implies $\nor{2}$, let $\seq {p}{n}$ be an enumeration of all the elements of $F[x]$
whose images under $\psi$ are finite products of linear terms in $K[x]$.  This list consists of all those polynomials over $F$ whose images under $\psi$
split completely in $K$.  Since $\nor{1}$ holds, the splitting field of the images under $\psi$ of this sequence of polynomials is a subfield
of $K$.  Also, if $a \in K$, then the minimal polynomial of $a$ is $\psi(p_n)$ for some $n$.  Thus, $K$ is equal
to the splitting field of the images under $\psi$ of the sequence of polynomials.

To see that $\nor{2}$ implies $\nor{3}$, suppose $\nor{2}$ holds.  Let $\langle \clo K , \tau \rangle$ be an algebraic closure
of $K$, and let $\varphi :K \to \clo K$ be an $F$-embedding.  If $p(x) \in F[x]$ is a defining polynomial of $K$ and $\alpha$
is any root of $\psi (p)$, then there must be a root $\beta$ of $\psi (p )$ such that $\varphi (\alpha ) = \tau (\beta)$.
Since every element of $K$ is expressible as a sum of products of these roots, $\varphi$ must map $K$ into
the image of $K$ in $\clo K$ under $\tau$.  Thus we can find an automorphism $\varphi^* : K \to K$ such that
for all $k \in K$, $\varphi (k) = \tau (\varphi^* (k))$.  Since $\varphi$ fixes $F$, so does the restriction $\varphi^*$.

To see that $\nor{3}$ implies $\nor{4}$, suppose that $\varphi$ is an $F$-automorphism of $\clo K$.  Then the
restriction of $\varphi$ to $K$ is an $F$-embedding of $K$ into $\clo K$.  By $\nor{3}$, this restriction is an
$F$-automorphism of $K$, as desired.

To prove the penultimate sentence of the theorem, we will work in $\rca$, assume that $F\subset K\subset \clo K$,
and prove that the negation of  $\nor{1}$ implies the negation of $\nor{4}$.  Let $p$ be a polynomial irreducible over $F$ that does not
split in $K$ but has a root $\alpha$ in $K$.  Let $\beta$ be a root of $p$ not lying in $K$.  By Theorem \ref{hungerford18}
there is an $F$-isomorphism $\varphi: F(\alpha ) \to F ( \beta )$.  By the last sentence of Theorem \ref{diamond}, $\varphi$
extends to an $F$-automorphism of $\clo K$.  The restriction of $\varphi$ to $K$ maps $\alpha$ to $\beta$, so it is
not an $F$-automorphism of $K$.  Thus, $\nor{4}$ fails as desired.

To prove the final sentence of the theorem, we continue working in $\rca$. Assume that $F \subset K\subset \clo K$ and $\nor{4}$ holds.
Suppose $\alpha \in K\setminus F$.  Let $p$ be the minimal polynomial of $\alpha$ over $F$ and apply the separability of $F$
to find a root $\beta$ of $p$
that is not equal to $\alpha$.
By Theorem \ref{hungerford18}
there is an $F$-isomorphism $\varphi: F(\alpha ) \to F ( \beta )$.  By the last sentence of Theorem \ref{diamond}, $\varphi$
extends to an $F$-automorphism of $\clo K$.  By $\nor{4}$, this restricts to an $F$-automorphism of $K$ that moves $\alpha$.
So $K$ is a Galois extension of $F$.
\end{proof}

Each converse omitted from the preceding theorem is equivalent to $\wkl$.

\begin{thm}\label{normal2}\label{normalgalois}
$( \rca )$  The following are equivalent:
\begin{enumerate}
\item  $\wkl$.
\item  For every field $F$ and every algebraic extension $K$ of $F$, $\nor{4} \to \nor{1}$.
\item  For every field $F$ and every algebraic extension $K$ of $F$, $\nor{4} \to \nor{3}$.
\item  For every field $F$ and every algebraic extension $K$ of $F$, $\nor{3} \to \nor{1}$.
\item  For every field $F$ and every separable algebraic extension $K$ of $F$, $\nor{1} \to \gal$.\label{normal2v}
\end{enumerate}
In light of Theorem \ref{normal1}, the equivalences hold with $\nor{1}$ replaced by $\nor{2}$.
\end{thm}

\begin{proof}
To prove that (1) implies (2), we will use $\wkl$ and $\neg \nor{1}$ to deduce $\neg \nor {4}$.
Let $\langle K , \psi \rangle$ be an algebraic extension of $F$.
On the basis of $\neg \nor{1}$, let $p(x)$ be an irreducible polynomial in $F[x]$ such that $\alpha \in K$
is a root of $\psi(p)(x)$ and $\psi(p)(x)$ does not split completely over $K$.  Let $q(x)$ be a nonlinear irreducible factor of $\psi(p)(x)$
in $K[x]$, and let $\beta$ be a root of $q(x)$.  By Theorem \ref{hungerford18}, $F( \alpha ) \cong _ F F(\beta)$.
Using $\wkl$, we can apply Theorem \ref{diamond} and extend this isomorphism to an $F$-automorphism
of $\clo K$.  Since this automorphism does not restrict to an automorphism of $K$, we have $\neg \nor {4}$.

By Theorem \ref{normal1}, $\rca$ proves $\nor{1} \to \nor{3}$.  Thus $\rca$ proves that (2) implies (3).
Before dealing with (4), we will prove that (3) implies (1).  Our plan is to assume the contrapositive
of (3), that is that $\neg \nor{3} \to \neg\nor{4}$, and construct a separating set for the ranges of disjoint injections.
Let $f$ and $g$ be disjoint injections and without loss of generality, assume that $0$ is not in either of their ranges.
Suppose $\clo \rat$ is an algebraic closure of a real closure of $\rat$ in which the positive roots and the elements
$\root 4 \of 2$, $-\root 4 \of 2$, $i\root 4 \of 2$, and $-i\root 4 \of 2$ have been designated.  Using the notation for primes from
the reversal of Theorem \ref{diamond}, define
$F = \rat (\sqrt{\mathstrut p_{f(i)}} , \sqrt{\mathstrut 2p_{g(i)}} \mid i \in \nat )$ and consider $F(\root 4 \of 2 )$.
$\rca$ proves that the usual $F$-isomorphism from $F(\root 4 \of 2 )$ to $F(i\root 4 \of 2 )$ exists and that
it is an embedding of $F(\root 4 \of 2 )$ into $\clo F$ which is not an automorphism of $F(\root 4 \of 2 )$.
Since $\neg \nor{3}$ holds, we may apply $\neg \nor{4}$ to find an $F$-automorphism of $\psi$ of $\clo F$
which maps some element of $F(\root 4 \of 2 )$ to an element not in $F(\root 4 \of 2 )$.  Thus
$\psi (\root 4 \of 2 )= \pm i \root 4 \of 2$ and so $\psi (\sqrt 2 ) = - \sqrt 2 $.
As in the reversal of Theorem \ref{diamond},
$S = \{ i \mid \psi ( \sqrt{p_i} ) = \sqrt{p_i} \}$ is a separating set for the ranges of $f$ and $g$.

Consider item (4).  Since Theorem \ref{normal1} shows $\nor{3} \to \nor{4}$ and by (2), $\wkl$ implies
that $\nor{4} \to \nor{1}$, $\wkl$ implies (4).  To prove the converse, we will use $\neg \nor{1} \to \neg \nor{3}$ to
find a separating set for the ranges of disjoint injections with nonzero ranges.
Let $f$, $g$, and $F$ be as in the preceding paragraph and let $K = F( \root 4 \of 2 )$.
The polynomial $x^4 -2$ has a root in $K$, but $x^4 - 2$ does not split
in $K$, since $i\root 4 \of 2$ is not in $K$.
Since $\neg \nor{1}$ holds for $F$ and $K$, by the contrapositive of (4), $\neg \nor {3}$ holds.
Let $\psi: K \to \clo K$ be an $F$-embedding which maps some element of $K$ outside $K$.
Then $\psi ( \root 4 \of 2 ) = \pm i\root 4 \of 2$, so
$\psi ( \sqrt{2} ) = - \sqrt{2}$ and $S = \{i | \psi(\sqrt{p_i}) = \sqrt{p_i} \}$ is a separating set.

The equivalence of $\wkl$ and (\ref{normal2v}) is immediate from part (\ref{mn3p}) of Theorem~\ref{mn}, using terminology from
Definition \ref{defngalois}.
\end{proof}

We conclude this section by recasting Theorem \ref{tower1} using normal field extensions.  The resulting formulation avoids root moduli, but is interestingly weaker
than the unbounded statement in Theorem \ref{tower2}.

\begin{thm}\label{normaltower}
$( \rca )$
The following are equivalent.
\begin{enumerate}
\item \textup{$\wkl$}.

\item  Suppose that $J$ and $K$ are $\nor{1}$ algebraic extensions of $F$.  If $F(k) \preceq_F J$ for all $k \in K$
 then $K \preceq_F J$.

\end{enumerate}
Moreover, the equivalence holds if $\nor{1}$ is replaced by $\nor{2}$, $\nor{3}$, or $\nor{4}$.
If $K$ and $J$ are separable extensions, then the equivalence holds if $\nor{1}$ is replaced by $\gal$.
\end{thm}

\begin{proof}
The proof follows from two simple observations.  Given $\nor{1}$ field extensions as in (2), $\rca$ can prove the existence of $F$-embedding bounds of $J$ into $K$ and of $K$ into $J$.  The forward
implication follows immediately from Theorem \ref{tower1}.  The proof of the reversal of Theorem \ref{tower1} also proves this reversal,  since every algebraic closure
of $F$ satisfies $\nor{1}$.
\end{proof}

\section {Galois correspondence theorems}\label{sectiongalcor}

Lemma 2.11 of Friedman, Simpson, and Smith \cite{fss} shows that Galois correspondence for field extensions of finite degree is
provable in $\rca$.  In this section, we analyze Galois correspondence for infinite extensions.
If $\langle E, \psi \rangle$ is an algebraic extension of $F$ and $\langle K , \varphi \rangle$ is an algebraic extension of $E$, then
$\langle K , \varphi \circ \psi \rangle$ is an algebraic extension of $F$.  In this case we say $E$ is an intermediate extension between $F$ and $K$.
By Lemma \ref{sigma}, every $\Sigma^0_1\text{-}F$-subfield of $K$ is the isomorphic image of an intermediate extension field
between $F$ and $K$.

\begin{thm}\label{subgalois}$(\rca )$
  The following are equivalent:
  \begin{enumerate}
  \item $\wkl$
  \item If $K$ is a Galois extension of $F$ and $E$ is an intermediate extension, then $K$ is a Galois extension of $E$.
  \end{enumerate}
\end{thm}

\begin{proof}
By Theorem \ref{galoisnormal}, if $K$ is a Galois extension of $F$, then it is a $\nor{2}$-normal extension.
  It is easy to see that if $K$ is a $\nor{2}$-normal extension of $F$ and $E$ is an intermediate extension, then $K$ is necessarily a $\nor{2}$-normal extension of $E$.
  Therefore, (1) implies (2) by Theorem~\ref{normalgalois}.

  The fact that (2) implies (1) follows from the reversal of Theorem~\ref{mn}.
  The field $F$ constructed there is strictly intermediate between $\clo{\rat}$ and $\rat$.
  It is not hard to see that $\clo{\rat}$ is a Galois extension of $\rat$.
  By (2), $\clo{\rat}$ is a Galois extension of $F$, so there must be a $F$-automorphism of $\clo \rat$ that is not the identity.  As in the
  proof of Theorem~\ref{mn}, this automorphism encodes the desired separating set.
\end{proof}

We now turn to the group-theoretic aspects of Galois theory.
The group $\Sym$ of permutations of $\nat$ has a topology which makes it into a complete separable metric space with respect to the distance $$d(\varphi,\psi) = \inf\{2^{-n} : (\forall i < n)(\varphi(i) = \psi(i) \land \varphi^{-1}(i) = \psi^{-1}(i))\}.$$
Note that composition and inversion are both continuous operations with respect to this topology.
Furthermore, $\Sym$ is easily understood even in $\rca$ with the usual representation of complete metric spaces in subsystems of second-order arithmetic.  See section II.5 of Simpson's book \cite{simpson}.

If $F$ is a subfield of $K$, the class $\Aut(K/F)$ of $F$-automorphisms of $K$ corresponds to a closed subgroup of $\Sym$.
Indeed, if $\varphi$ is a permutation of $K$ which is not an $F$-automorphism, then there is a finite initial segment of $\varphi$ that cannot be extended to an $F$-automorphism of $K$.
The Galois correspondence says that there is an inclusion-reversing correspondence between intermediate fields $F \subset E \subset K$ and closed subgroups of $\Aut(K/F)$; this correspondence is provable in $\wkl$.

\begin{thm}$(\wkl)$\label{weakgalois}
  \emph{(Galois Correspondence.)}
  Suppose $K$ is a Galois extension of $F$.
  \begin{itemize}
  \item For every intermediate extension $E$ between $F$ and $K$, $K$ is a Galois extension of $E$, and $\Aut(K/E)$ is a closed subgroup of $\Aut(K/F)$.
  \item For every closed subgroup $H$ of $\Aut(K/F)$, there is an 
 intermediate extension $E$
 such that $K$ is a Galois extension of $E$,
  and $H = \Aut (K/E)$.
  \end{itemize}
\end{thm}

\begin{proof}
  The first part of the theorem is immediate from Theorem~\ref{subgalois}, but the second part requires proof.
  
  The first observation is that $\Aut(K/F)$ is a bounded subgroup of $\Sym$.
  Indeed, since $K$ is a normal extension of $F$, for every $k \in K$, we can effectively find a polynomial $p_k(x) \in F[x]$ such that $p_k(k) = 0$ and $p_k(x)$ splits completely in $K$.
  Consequently, $\rca$ proves the existence of an $F$ embedding bound, $b: K \to K^{<\nat}$.  Any $F$-automorphism of $K$ must send $k$ to some element of $b(k)$.
  By the last sentence of Definition \ref{feb}, if $\varphi$ is an $F$-automorphism of $K$, then $\varphi (k) \le b(k)$ for all $k\in K$.

Applying $\Delta^0_1$-comprehension, we can prove the existence of a $b$-bounded tree of initial segments of elements of $\Aut(K/F)$.
Briefly, given an enumeration $\langle k_i \rangle_{i\in \nat}$ of $K$, place $\sigma$ in the tree if for all $i,j < \lth (\sigma )$ we have (1) $\sigma (i) \le b(k_i )$,
(2) if $j$ witnesses that $k_i \in F$ then $\sigma(i) = k_i$, and (3) $\sigma$ preserves field operations.
  A closed subgroup $H$ of $\Aut(K/F)$ corresponds to branches through a $b$-bounded subtree $T_H$.
  By $\wkl$, an element $k_n$ of $K$ is fixed by every automorphism in $H$ if and only if there is a level $m > n$ such that every element of $T_H \cap \nat^{m}$ fixes $n$.
  Since $T_H$ is $b$-bounded, this is a $\Sigma^0_1$ definition of the fixed field $K^H$.  By Lemma \ref{sigma}, there is an isomorphic intermediate extension
  $\langle E, \tau \rangle$.  By Theorem \ref{subgalois}, $K$ is a Galois extension of $E$.

  It remains to see that $H = \Aut(K/E)$.
  The inclusion $H \subset \Aut(K/E)$ is clear, so suppose that $\psi$ is an $E$-automorphism of $K$.
  We need to show that every initial segment of $\psi$ is in the tree $T_H$.
Let $p(x)$ be a polynomial in $E(x)$ such that $\tau (p)$ splits in $K$ and the roots of $\tau (p)$ include
$k_0 , \dots , k_{n-1}$.  Let $L $ be the splitting field of $\tau (p)$.
  Then $\psi$ restricts to an $E$-automorphism $\clo{\psi}$ of $L$.
  Every element $\varphi$ of $H$ also restricts to an $E$-automorphism $\clo{\varphi}$ of $L$ and these restrictions form a group $\clo{H}$ of automorphisms of $L$.
  Furthermore,
  $E$ is the subfield of $L$ fixed by $\clo H$ since $E$ is the subfield of $K$ fixed by $H$.
  It follows from finite Galois theory that $\clo{H} = \Aut(L/E)$ \cite{fss}*{Lemma~2.11}, which means that $\clo{\psi} = \clo{\varphi}$ for some $\varphi \in H$.
  Since $k_0,k_1,\dots,k_{n-1} \in L$, it follows that $\psi(m) = \varphi(m)$ for all $m < n$ and hence that the initial segment of $\psi$ with length $n$ belongs to $T_H$.
\end{proof}

\noindent
We already saw in Theorem~\ref{normalgalois} that the first part of the Galois correspondence requires $\wkl$ (though $\Aut(K/E)$ is always a closed subgroup of $\Aut(K/F)$).
In the second part of the correspondence theorem, $E$ is essentially the fixed field for $H$, and the fixed field associated with a closed subgroup of $\Aut(K/F)$ is difficult
to define in subsystems weaker than $\wkl$.

Although $\Aut(K/F)$ is always a closed subgroup of $\Sym$, this does not mean that $\Aut(K/F)$ is a complete separable metric space like $\Sym$.
Indeed, $\Aut(K/F)$ could fail to have a countable dense subset.  The following definitions are related to those of Brown \cite{brown}.

\begin{defn}$(\rca)$
  Let $F$ be a subfield of $K$.
  \begin{itemize}
  \item We say  $\Aut(K/F)$ is {\sl separably closed} if there is a sequence $\seq{\varphi}{i}$ of elements of $\Aut(K/F)$ such that for every $\psi \in \Aut(K/F)$ and every $n \in \nat$, there is an $i \in \nat$ such that $d(\varphi_i,\psi) \leq 2^{-n}$.
  \item We say $\Aut(K/F)$ is {\sl separably closed and totally bounded} if there is a sequence $\seq{\varphi}{i}$ of elements of $\Aut(K/F)$ and a function $b:\nat\to\nat$ such that for every $\psi \in \Aut(K/F)$ and every $n \in \nat$, there is an $i \leq b(n)$ such that $d(\varphi_i,\psi) \leq 2^{-n}$.
  \end{itemize}
\end{defn}

When $\Aut(K/F)$ is separably closed, this group can also be understood using the usual representation of complete metric spaces in second-order arithmetic.
However, this is not always the case unless we assume $\aca$ (in which case every closed subgroup of $\Sym$ is separably closed).

\begin{lemma}$(\rca)$\label{sepgalois}
  Suppose $K$ is a Galois extension of $F$. Then the following are equivalent:
  \begin{enumerate}
  \item $\Aut(K/F)$ is separably closed and totally bounded.
  \item $\Aut(K/F)$ is separably closed.
  \item $F$ is a subset of $K$ fixed by its embedding.
  \end{enumerate}
\end{lemma}

\begin{proof}
  It is clear that (1) implies (2).

  To see that (2) implies (3), suppose that $\seq{\varphi}{i}$ enumerates a dense set of elements of $\Aut(K/F)$.
  We claim that $$\alpha \in F \biimp (\forall i)(\varphi_i(\alpha) = \alpha).$$
  Since the displayed formula is $\Pi^0_1$, this shows that $F$ is a $\Delta^0_1$ subset of $K$.
  Since $\seq{\varphi}{i}$ consists of elements of $\Aut(K/F)$, the forward implication is clear.
  For the converse, suppose $\alpha$ is an element of $K$ that is not in $F$.
  Then, since $K$ is a Galois extension of $F$, there is an $F$-automorphism $\varphi$ of $K$ such that $\varphi(\alpha) \neq \alpha$.
  By density, there is an $i$ such that $\varphi_i(\alpha)  = \varphi(\alpha)$ and so $\varphi_i(\alpha) \neq \alpha$.

To see that (3) implies (1), assume that $F$ is a set.  Given the first $n$ elements of $K$, by Lemma 2.8 of Friedman, Simpson, and Smith \cite{fss}
we can find polynomials irreducible over $F$ corresponding to each element and the roots of these polynomials in $K$.
From these construct the finite list of all possible related initial segments of $F$-automorphisms of $K$.  Emulating the construction at the end of the
proof of Theorem \ref{diamond}, we can extend these to $F$-automorphisms of $K$.  For every $\psi \in \Aut(K/F)$ there
will be a $\varphi$ in this collection such that $d(\psi,\varphi ) \le 2^{-n}$.  This construction can be carried out uniformly, yielding the sequence
and function witnessing that $\Aut(K/F)$ is separably closed and totally bounded.
\end{proof}

\begin{thm}$(\rca)$\label{stronggalois}
  \emph{(Strong Galois Correspondence.)}
  Suppose $K$ is a Galois extension of $F$.
  \begin{itemize}
  \item For every set $E$ which is a field that contains $F$ and is contained in $K$, $K$ is a Galois extension of $E$, and $\Aut(K/E)$ is a separably closed and totally bounded
  subgroup of $\Aut(K/F)$.
  \item For every separably closed and totally bounded subgroup $H$ of $\Aut(K/F)$, the collection $E$ of elements fixed by $H$ is a set contained in $K$,
  $K$ is a Galois extension of $E$, and $H = \Aut(K/E)$.
  \end{itemize}
\end{thm}

\begin{proof}
  The first part of the theorem follows from the last sentence of Theorem~\ref{galoisnormal} and Lemma~\ref{sepgalois}.

  For the second part of the theorem, suppose that $\seq{\varphi}{i}$ and $b:\nat\to\nat$ witness that $H$ is  separably closed and totally bounded.
  Then, the subfield E of K fixed by H can be defined by the bounded formula $(\forall i \leq b(k))(\varphi_i(k) = k)$, which therefore exists by $\Delta^0_1$-comprehension.

  It remains to see that $H = \Aut(K/E)$.
  The inclusion $H \subset \Aut(K/E)$ is clear, so suppose that $\psi$ is an $E$-automorphism of $K$.
  Pick $n$ elements $\{k _0 , \dots , k_{n-1} \}$ of $K$.  
  Let $L $
  be the normal closure of $E(k_0,\dots,k_{n-1})$.
  (That is, $L$ is the splitting field for the minimal polynomials of $k_0 , \dots , k_{n-1}$.)
  For each $i < n$, let $p_i\in F[x]$ be a polynomial with root $k_i$ that splits into linear factors in $K$, and let $m$ be the largest root of these polynomials.
  Now $\psi$ restricts to an $E$-automorphism $\clo{\psi}$ of $L$.
  Every $\varphi_i$ also restricts to an $E$-automorphism $\clo{\varphi}_i$ of $L$ and the first $b(m)+1$ such restrictions actually form a group $\clo{H} = \{\clo{\varphi}_0,\dots,\clo{\varphi}_{b(m)}\}$ of automorphisms of $L$.
  Furthermore,
  $E$ is the subfield of $L$ fixed by $\clo H$ since $E$ is the subfield of $K$ fixed by $H$.
  It follows from finite Galois theory that $\clo{H} = \Aut(L/E)$ \cite{fss}*{Lemma~2.11}, which means that $\clo{\psi} = \clo{\varphi}_i$ for some $i \leq b(m)$.
  Since $k_0,k_1,\dots,k_{n-1} \in L$, it follows that $d(\varphi_i ,\psi) \leq 2^{-n}$.
  Since this holds for every $n \in \nat$ we see that $\varphi \in H$.
\end{proof}

Galois theory also says that if $K$ is a Galois extension of $F$ and $L$ is an intermediate field, then $L$ is a Galois extension of $F$ if and only if $\Aut(K/L)$ is a normal subgroup of $\Aut(K/F)$, in which case $\Aut(L/F)$ is isomorphic to the quotient group $\Aut(K/F)/\Aut(K/L)$.
To analyze this, we first prove a variant of Theorem~\ref{normalgalois} in $\rca$.

\begin{thm}$(\rca)$\label{intgalois}
  Let $K$ be a Galois extension of $F$ and let $L$ be an intermediate extension.
  The following are equivalent:
  \begin{enumerate}
  \item $L$ is a Galois extension of $F$.
  \item $L$ is a $\nor{1}$-normal extension of $F$.
  \item $L$ is a $\nor{2}$-normal extension of $F$.
  \item If $\varphi: L \to K$ is an $F$-embedding, then $\varphi$ is an $F$-automorphism of $L$.  (This is a variant of $\nor{3}$.)
  \item Every $F$-automorphism of $K$ restricts to an $F$-automorphism of $L$.  (This is a variant of $\nor{4}$ and uses the notion of restriction from Definition \ref{defnext}.)
  \end{enumerate}
\end{thm}

\begin{proof}
 Theorem~\ref{galoisnormal} shows that (1) implies (2) and that (2) implies (3).
 The proof that (3) implies (4) is analogous to the proof that $\nor{2}$ implies $\nor{3}$ in Theorem \ref{galoisnormal}.
 The proof that (4) implies (5) is analogous to the proof that $\nor{3}$ implies $\nor{4}$ in Theorem \ref{galoisnormal}.
  Since $K$ is a Galois extension of $F$ it follows immediately that (5) implies (1).
\end{proof}

The next theorem uses the following terminology.  If $G$ is a class that is a group and $N$ is a subclass that is also a group, we say that
$N$ is a {\sl normal subgroup} of $G$ if for all $\varphi \in N$ and $\psi \in G$, $\psi\varphi\psi^{-1}$ is in $N$.

\begin{thm}$(\rca)$
  Let $K$ be a Galois extension of $F$ and let $L$ be an intermediate extension. 
  \begin{enumerate}
  \item If $L$ is a Galois extension of $F$ then $\Aut(K/L)$ is a normal subgroup of $\Aut(K/F)$.
  \item If $K$ is a Galois extension of $L$ and $\Aut(K/L)$ is a normal subgroup of $\Aut(K/F)$ then $L$ is a Galois extension of $F$.
  \item If $L$ is  also a subset of $K$, then $\Aut(K/L)$ is a normal subgroup of $\Aut(K/F)$ if and only if $L$ is a Galois extension of $F$.
  \end{enumerate}
\end{thm}

\begin{proof}
  For the first statement,
  suppose $\varphi$ is an element of $\Aut (K/L)$ and $\psi$ is an element of $\Aut (K/F)$.  Then $\psi^{-1}$ is also in $\Aut(K/F )$.
  Consider $\psi \varphi \psi^{-1}$ and let $x \in L$.  Since $L$ is a Galois extension of $F$, by part (5) of Theorem \ref{intgalois}, $\psi^{-1} (x) \in L$.
  Thus $\varphi(\psi^{-1} (x)) = \psi^{-1} (x)$ and $\psi (\varphi \psi^{-1} (x)) = x$.  Thus $\psi \varphi \psi^{-1} \in \Aut (K/L)$ and
  so $\Aut (K/L)$ is a normal subgroup of $\Aut (K/F )$.

  For the second statement, a simple algebraic computation shows that if $\varphi$ is an $F$-automorphism of $K$, then $\Aut(K/\varphi[L]) = \varphi\Aut(K/L)\varphi^{-1}$.
  If $\Aut(K/L)$ is a normal subgroup of $\Aut(K/F)$ then $\varphi\Aut(K/L)\varphi^{-1} = \Aut(K/L)$.
  Assuming that $K$ is Galois over $L$, it follows that $L = \varphi[L]$ and hence that $\varphi$ restricts to an automorphism of $L$.
  By part~(5) of Theorem~\ref{intgalois}, it follows that $L$ is a Galois extension of $F$.

  The last statement follows from the previous two and Theorem~\ref{stronggalois} which shows that $K$ is necessarily a Galois extension of $L$.
\end{proof}

\noindent
Informally, if $L$ is an intermediate Galois extension of $F$, then
the restriction map from $K$ to $L$ takes each element of $\Aut(K/F)$ and restricts its domain to create an automorphism of $L$.
Consequently, the restriction map as described in part (5) of Theorem \ref{intgalois}
is a homomorphism from $\Aut(K/F)$ to $\Aut(L/F)$ whose kernel is $\Aut(K/L)$.
However, the homomorphism from $\Aut(K/F)$ to $\Aut(L/F)$ needs to be surjective in order to conclude that $\Aut(L/F)$ is isomorphic to the quotient of $\Aut(K/F)$ by $\Aut(K/L)$, which we can't really talk about in second-order arithmetic other than via the First Isomorphism Theorem.

\begin{thm}\label{galoisquot}$(\rca)$
  The following are equivalent:
  \begin{enumerate}
  \item $\wkl$
  \item If $K$ is a Galois extension of $F$ and $L$ is an intermediate extension of $F$, then the restriction map is a surjective homomorphism from $\Aut(K/F)$ onto $\Aut(L/F)$ whose kernel is $\Aut(K/L)$.
  \end{enumerate}
If $L$ is a subset of $K$ fixed by its embedding, then (2) is provable in $\rca$.
\end{thm}

\begin{proof}{}
Note that (2) simply states that any $F$-automorphism of $L$ can be extended to an $F$-automorphism of $K$. The proof is similar to that of Theorem \ref{diamond}.
\end{proof}

\section{Other characteristics}\label{sectionotherchar}

Results in sections \S\ref{sectionauto} and \S\ref{sectioniso} can be extended to fields of finite characteristic.  In many cases, separability conditions
must be appended to the hypotheses.  Additionally, when the characteristic is specified in the result, any
reversal must reflect this.  The final result of this section, based on Theorem \ref{mn}, illustrates the adaptation process.

Many of the reversals in previous sections involve extensions of $\rat$.  Adaptation of these arguments relies on the following observation.
Let $p$ be a prime and let $\gfpn$ denote the field of integers mod $p^n$.  The field of rational functions $\gfpn (x)$ is an infinite field of characteristic
$p$ and is the quotient field of the Euclidean ring $\gfpn [x]$.  Because $\gfpn$ is finite, $\rca$ can prove the existence of the set of
monic irreducible polynomials of $\gfpn[x]$.  These irreducible polynomials can play the role the prime numbers in our prior constructions.
For example, we have the following versions of Lemma \ref{roththm}.

\begin{lemma}\label{rothalt}
$(\rca)$   Let $R$ be a Euclidean ring with quotient ring $Q$ of characteristic not equal to $2$.
Let $p_1 , \dots , p_n$ and $q_1 , \dots , q_r$ be disjoint lists of distinct primes (irreducible elements).  Then
\[
\sqrt {q_1 \dots q_r} \notin Q ( \sqrt{p_1}, \dots , \sqrt{p_n } )\text{~ and~}
\sqrt {q_1} \notin Q (\sqrt{p_1} , \dots , \sqrt {p_n}, \sqrt {q_1 q_2} ,\dots, \sqrt{q_1 q_r} ).
\]
\end{lemma}

\begin{proof}
We will work in $\rca$.
Fix $R$.  Note that the first conjunct of the conclusion can be written as: for every $n$, for every list of $p$s, for every list of $q$s, for every
quotient of $Q$-linear combinations of products of roots of $p$s, the square of the linear combination is not equal to the product of the $q$s.
Since this conjunct can be expressed as a $\Pi^0_1$ formula, we can proceed to prove it in $\rca$ by induction on $n$.

For the base case, suppose by way of contradiction that $\sqrt{q_1 \dots q_r} \in Q$.  Let $\sqrt{q_1 \dots q_r} = \frac{r_0}{r_1}$ where $r_0, r_1 \in Q$
and gcd$(r_0 , r_1 )= 1$.  Thus $r_1^2 q_1 \dots q_r = r_0^2$.  Since $q_1$ is prime and $q_1 | r_0^2$, we have $q_1 | r_0$.
So $r_0^2 = q_1^{2m} r_2$ where $m\ge 1$ and gcd$(q_1 , r_2 ) = 1$.  Since
$q_1^2 | r_1 ^2 q_1 \dots q_r$ and $q_1 , \dots , q_r$ are distinct primes, $q_1 | r_1^2$.  Thus $q_1 | r_1$ and so $r_1^2 q_1 \dots q_r = q_1^{2k+1} r_3$ where
$k\ge 1$ and gcd$(q_1 , r_3 ) = 1$.  Summarizing, $q_1^{2k+1} r_3 = q_1^{2m} r_2$ where $q_1 \not |~ r_3$ and $q_1 \not |~ r_2$, a contradiction.

For the induction step, suppose the lemma is true for $n-1$.  Fix distinct primes $p_1 , \dots , p_n$.  Let $F_0 = Q ( \sqrt{p_1} , \dots , \sqrt{p_{n-1}} )$.
Let $q_1 , \dots , q_r$ be a list of distinct primes disjoint from $p_1 , \dots , p_n$.  Suppose by way of contradiction that
$\sqrt{q_1 \dots q_r } \in F_0 ( \sqrt {p_n} )$.  Then we may write $\sqrt {q_1 \dots q_r} = \alpha + \beta \sqrt{p_n}$ where $\alpha , \beta \in F_0$.  Squaring yields
$q_1 \dots q_r = \alpha^2 + \beta^2 p_n + 2 \alpha \beta \sqrt{p_n}$.
Consider three cases:  (1)  If $\alpha \beta \neq 0$ then $\sqrt{p_n } \in F_0$, contradicting the induction hypothesis.  (2)  If $\beta = 0$ then
$\sqrt{q_1 \dots q_r} = \alpha \in F_0$, contradicting the induction hypothesis.  (3)  If $\alpha = 0$ then
$\sqrt{q_1 \dots q_r }= \beta \sqrt {p_n }$ so $\sqrt{q_1 \dots q_r p_n } = p_n \beta \in F_0$, contradicting the induction hypothesis.

This completes the induction proof of the first conjunct of the conclusion of the lemma.  The remaining conjunct is proved by the same
argument as Lemma \ref{rothcoro}.
\end{proof}

\begin{lemma}\label{carro}
$(\rca )$  Let $\{ p_i \mid i \le n \}$ be a sequence of distinct irreducible elements of ${\sf{GF}} (4) [x]$.
For each $i \le n$, let $r_i$ be a solution of $x^3 - p_i = 0$.  Then the set
$A = \{ \prod_{i \le n } r_i^{\varepsilon_i} \mid \forall i ~0 \le \varepsilon_i \le 2 \}$ is linearly independent over
${\sf{GF}} (4) (x)$.
Consequently, if $Q = {\sf{GF}} (4) (x)$ and
$p_1 , \dots , p_n$ and $q_1 , \dots , q_r$ are disjoint lists of distinct primes, then
\[
\root 3 \of{q_1 \dots q_r} \notin Q ( \root 3 \of{p_1}, \dots , \root 3 \of{p_n } )\text{~ and~}
\root 3 \of {q_1} \notin Q (\root 3 \of {p_1} , \dots , \root 3 \of {p_n}, \root 3 \of {q_1 q_2} ,\dots ,\root 3 \of{q_1 q_r} ).
\]
\end{lemma}

\begin{proof}
A straightforward algebraic argument proves that $A$ is pairwise linearly independent over  ${\sf{GF}} (4) (x)$.
The first sentence of the lemma follows from Theorem 1.3 of Carr and O'Sullivan \cite{co}, substituting ${\sf{GF}} (4) (x)$ for their $K$,
$\clo K$ for $L$, and $A$ (as in the statement) for $A$.  This instance of their theorem can be proved in $\rca$.
The remainder of the lemma can be proved in much the same fashion as Lemma \ref{rothalt}.
\end{proof}

\begin{thm}\label{charp}
$(\rca )$  Let $p$ be a prime or $0$.  The following are equivalent:
\begin{enumerate}
\item  $\wkl$.
\item  Let $F$ be an infinite field of characteristic $p$ and let $K$ be an algebraic extension of $F$ that includes a separable element $\alpha \notin F$.
Then there is an $F$-embedding of $K$ into $\clo K$ that is not the identity.
\end{enumerate}
\end{thm}

\begin{proof}
To prove that (1) implies (2), assume $\wkl$.  Since $\alpha$ is separable, it is a root of a polynomial $p(x) \in F[x]$ with no repeated roots.
Since $\alpha \notin F$, the degree of $p(x)$ is greater than $1$.  Let $\beta \neq \alpha$ be another root of this polynomial.  Imitate the
proof of Theorem \ref{mn}.  Since the proof of Theorem \ref{diamond}
does not rely on the characteristic of $F$, it can be used to complete the proof.

Next, we will prove the reversal for characteristic $0$, and then adapt the argument for other characteristics.
Let $f$ and $g$ be injections such that $\forall s \forall t (0 \neq f(s) \neq g(t)\neq 0)$.  As in the proof of Theorem \ref{mn}, let $(i,j)$ denote both the ordered
pair and the integer code for that ordered pair.  Let $p_i$ denote the $i^{\text{th}}$ prime.  Define the fields $F$ and $K$ by:
\[
F = \rat ( \sqrt{p_{(i,f(j))}}, \sqrt{p_i p_{(i,g(j))}} \mid i,j \in \nat ) \qquad 
K = \rat ( \sqrt {p_i} \mid i \in \nat )
\]
By Lemma \ref{rothcoro}, $\sqrt{2}$ is not an element of $F$, so $K$ is a nontrivial extension of $F$.
Suppose $\varphi$ is a nontrivial $F$-embedding of $K$ into $\clo K$.  Then for some prime $p_i$, $\varphi ( \sqrt{p_i } ) \neq \sqrt{p_i}$.
For this $i$ and any $j$, $\varphi (\sqrt{p_{(i,g(j))}}) \neq \sqrt{p_{(i,g(j))}}$ and $\varphi ( \sqrt{p_{(i,f(j))}}) = \sqrt{p_{(i,f(j))}}$.
The separating set $S = \{ k \mid \varphi(\sqrt{p_{(i,k)}} ) = \sqrt{p_{(i,k)}} \}$ exists by $\Delta^1_0$ comprehension using the
parameter $\varphi$.  Since $S$ includes the range of $f$ and avoids the range of $g$, this proves $\wkl$.

Now suppose $p$ is an odd prime and (2) holds for fields of characteristic $p$.  Our goal is to adapt the previous construction to the
characteristic $p$ setting.  Let $\{ p_i \mid i \in \nat \}$ be a list of distinct irreducible monic polynomials in $\gfp [x]$. 
These will play the role that the prime numbers played in the preceding argument.
For each $p_i$, the polynomial $z^2 - p_i$ and its derivative have no common roots, so $z^2 - p_i$ is separable.
Let $r_i$ denote a root of $z^2 - p_i$.  Given
disjoint injections $f$ and $g$ that never take the value $0$, define the fields $F$ and $K$ by
\[
F=\gfp(x)(r_{(i,f(j))} , r_i r_{(i,g(j))} \mid i,j \in \Nb )\qquad
K=\gfp (x) (r_i \mid i \in \Nb )
\]
By Lemma \ref{rothalt}, $K$ is a nontrivial extension of $F$.
To complete the proof, use a nontrivial $F$-embedding of $K$ to find a separating set for the ranges of $f$ and $g$.

To carry out the reversal for characteristic $2$, modify the previous argument by using ${{\sf GF }}(4)(x)$, $z^3 - p_i$, and Lemma \ref{carro}.
\end{proof}

Some of the reversals in previous sections use algorithms for factoring polynomials over $\rat$.
One can find
factoring algorithms for the characteristic $p$ fields used in this section by adapting work of Stoltenberg-Hansen and Tucker \cite{stoltuck}.

\section*{Acknowledgements}

The authors would like to thank Bill Cook for useful discussions, and the referees for their helpful comments and suggestions.
Portions of Jeffry Hirst's work were supported by a grant (ID\#20800) from the John Templeton Foundation. 
The opinions expressed in this publication are those of the authors and do not necessarily reflect the views of the John Templeton Foundation.
Paul Shafer's work was funded in part by an FWO Pegasus Long Postdoctoral Fellowship.


\begin{bibsection}[\bibname]
\label{bib}
\begin{biblist}

\bib{artin}{book}{
   author={Artin, Emil},
   title={Galois theory},
   edition={2},
   note={Edited and with a supplemental chapter by Arthur N. Milgram},
   publisher={Dover Publications Inc.},
   place={Mineola, NY},
   date={1998},
   pages={iv+82},
   isbn={0-486-62342-4},
   review={\MR{1616156 (98k:12001)}},
}

\bib{brown}{article}{
   author={Brown, Douglas K.},
   title={Notions of closed subsets of a complete separable metric space in weak subsystems of second-order arithmetic},
   conference={
      title={Logic and computation},
      address={Pittsburgh, PA},
      date={1987},
   },
   book={
      series={Contemp. Math.},
      volume={106},
      publisher={Amer. Math. Soc.},
      place={Providence, RI},
   },
   date={1990},
   pages={39--50},
   note={DOI 10.1090/conm/106/1057814},
   review={\MR{1057814 (91i:03108)}},
}

\bib{co}{article}{
   author={Carr, Richard},
   author={O'Sullivan, Cormac},
   title={On the linear independence of roots},
   journal={Int. J. Number Theory},
   volume={5},
   date={2009},
   number={1},
   pages={161--171},
   issn={1793-0421},
   review={\MR{2499028 (2010c:11041)}},
   doi={10.1142/S1793042109002018},
}

\bib{fss}{article}{
   author={Friedman, Harvey M.},
   author={Simpson, Stephen G.},
   author={Smith, Rick L.},
   title={Countable algebra and set existence axioms},
   journal={Ann. Pure Appl. Logic},
   volume={25},
   date={1983},
   number={2},
   pages={141--181},
   issn={0168-0072},
   review={\MR{725732 (85i:03157)}},
   doi={10.1016/0168-0072(83)90012-X},
}

\bib{fss2}{article}{
   author={Friedman, Harvey M.},
   author={Simpson, Stephen G.},
   author={Smith, Rick L.},
   title={Addendum to: ``Countable algebra and set existence axioms'' [Ann.
   Pure Appl. Logic {\bf 25} (1983), no. 2, 141--181; MR0725732
   (85i:03157)]},
   journal={Ann. Pure Appl. Logic},
   volume={28},
   date={1985},
   number={3},
   pages={319--320},
   issn={0168-0072},
   review={\MR{790391 (87f:03141)}},
   doi={10.1016/0168-0072(85)90020-X},
}

\bib{hilbert}{article}{
  author={Hilbert, David},
  title={Die Theorie der algebraischen Zahlk\"{o}rper},
  journal={Jahresber. Deutsch. Math.-Verein},
  volume={4},
  date={1897},
  pages={175--546}
  }

\bib{realrep}{article}{
   author={Hirst, Jeffry L.},
   title={Representations of reals in reverse mathematics},
   journal={Bull. Pol. Acad. Sci. Math.},
   volume={55},
   date={2007},
   number={4},
   pages={303--316},
   issn={0239-7269},
   review={\MR{2369116 (2009j:03015)}},
   doi={10.4064/ba55-4-2},
}


\bib{hungerford}{book}{
author={Hungerford, Thomas},
title={Abstract Algebra},
series={Graduate Texts in Mathematics},
publisher={Springer Science+Business Media},
place={New York, NY},
date={1980},
pages={528},
isbn={978-0387905181},
}

\bib{js}{article}{
   author={Jockusch, Carl G., Jr.},
   author={Soare, Robert I.},
   title={$\Pi ^{0}_{1}$ classes and degrees of theories},
   journal={Trans. Amer. Math. Soc.},
   volume={173},
   date={1972},
   pages={33--56},
   issn={0002-9947},
   review={\MR{0316227 (47 \#4775)}},
}

\bib{kaplansky}{book}{
   author={Kaplansky, Irving},
   title={Fields and rings},
   series={Chicago Lectures in Mathematics},
   note={Reprint of the second (1972) edition},
   publisher={University of Chicago Press},
   place={Chicago, IL},
   date={1995},
   pages={x+206},
   isbn={0-226-42451-0},
   review={\MR{1324341 (96a:12001)}},
}


\bib{lang}{book}{
   author={Lang, Serge},
   title={Algebra},
   series={Graduate Texts in Mathematics},
   volume={211},
   edition={3},
   publisher={Springer-Verlag},
   place={New York},
   date={2002},
   pages={xvi+914},
   isbn={0-387-95385-X},
   review={\MR{1878556 (2003e:00003)}},
}

\bib{mn}{article}{
   author={Metakides, G.},
   author={Nerode, A.},
   title={Effective content of field theory},
   journal={Ann. Math. Logic},
   volume={17},
   date={1979},
   number={3},
   pages={289--320},
   issn={0003-4843},
   review={\MR{556895 (82b:03082)}},
   doi={10.1016/0003-4843(79)90011-1},
}

\bib{ms}{article}{
author={Miller, R.},
author={Shlapentokh, A.},
title={Computable categoricity for algebraic fields with spitting algorithms},
date={November 7, 2011},
note={Preprint, http://arxiv.org/pdf/1111.1205.pdf.}
}

\bib{remmel}{article}{
   author={Remmel, J. B.},
   title={Graph colorings and recursively bounded $\Pi^0_1$-classes},
   journal={Ann. Pure Appl. Logic},
   volume={32},
   date={1986},
   number={2},
   pages={185--194},
   issn={0168-0072},
   review={\MR{863333 (87m:03065)}},
   doi={10.1016/0168-0072(86)90051-5},
}

\bib{roth}{article}{
   author={Roth, R. L.},
   title={Classroom Notes: On Extensions of $Q$ by Square Roots},
   journal={Amer. Math. Monthly},
   volume={78},
   date={1971},
   number={4},
   pages={392--393},
   issn={0002-9890},
   review={\MR{1536291}},
   doi={10.2307/2316910},
}

\bib{simpson}{book}{
   author={Simpson, Stephen},
   title={Subsystems of second order arithmetic},
   series={Perspectives in Logic},
   edition={2},
   publisher={Cambridge University Press},
   place={Cambridge},
   date={2009},
   pages={xvi+444},
   isbn={978-0-521-88439-6},
review={\MR{2517689 (2010e:03073)}},
doi={10.1017/CBO9780511581007},
}

\bib{stoltuck}{article}{
   author={Stoltenberg-Hansen, V.},
   author={Tucker, J. V.},
   title={Computable rings and fields},
   conference={
      title={Handbook of computability theory},
   },
   book={
      series={Stud. Logic Found. Math.},
      volume={140},
      publisher={North-Holland},
      place={Amsterdam},
   },
   date={1999},
   pages={363--447},
   review={\MR{1720739 (2000g:03100)}},
   doi={10.1016/S0049-237X(99)80028-7},
}

\bib{vdw}{book}{
   author={van der Waerden, B. L.},
   title={Algebra. Vol. I},
   note={Based in part on lectures by E. Artin and E. Noether;
   Translated from the seventh German edition by Fred Blum and John R.
   Schulenberger},
   publisher={Springer-Verlag},
   place={New York},
   date={1991},
   pages={xiv+265},
   isbn={0-387-97424-5},
   review={\MR{1080172 (91h:00009a)}},
   doi={10.1007/978-1-4612-4420-2},
}

\bib{zariski}{book}{
   author={Zariski, Oscar},
   author={Samuel, Pierre},
   title={Commutative algebra. Vol. II},
   note={Reprint of the 1960 edition;
   Graduate Texts in Mathematics, Vol. 29},
   publisher={Springer-Verlag},
   place={New York},
   date={1975},
   pages={x+414},
   review={\MR{0389876 (52 \#10706)}},
}


\end{biblist}
\end{bibsection}
\end{document}